\documentclass[oneside,english,reqno]{amsart}
\usepackage[T1]{fontenc}
\usepackage[latin9]{inputenc}
\usepackage{geometry}
\usepackage{babel}
\usepackage{refstyle}
\usepackage{mathrsfs}
\usepackage{amsthm}
\usepackage{amstext}
\usepackage{amssymb}
\usepackage{graphicx}
\usepackage{esint}
\PassOptionsToPackage{normalem}{ulem}
\usepackage{ulem}
\usepackage[unicode=true,pdfusetitle,
 bookmarks=true,bookmarksnumbered=false,bookmarksopen=false,
 breaklinks=false,pdfborder={0 0 1},backref=false,colorlinks=false]
 {hyperref}
\hypersetup{
 colorlinks=true,citecolor=blue,linkcolor=blue,linktocpage=true}

\makeatletter

\AtBeginDocument{\providecommand\thmref[1]{\ref{thm:#1}}}
\AtBeginDocument{\providecommand\secref[1]{\ref{sec:#1}}}
\AtBeginDocument{\providecommand\propref[1]{\ref{prop:#1}}}
\AtBeginDocument{\providecommand\lemref[1]{\ref{lem:#1}}}
\AtBeginDocument{\providecommand\defref[1]{\ref{def:#1}}}
\AtBeginDocument{\providecommand\subref[1]{\ref{sub:#1}}}
\providecommand{\tabularnewline}{\\}
\RS@ifundefined{subref}
  {\def\RSsubtxt{section~}\newref{sub}{name = \RSsubtxt}}
  {}
\RS@ifundefined{thmref}
  {\def\RSthmtxt{theorem~}\newref{thm}{name = \RSthmtxt}}
  {}
\RS@ifundefined{lemref}
  {\def\RSlemtxt{lemma~}\newref{lem}{name = \RSlemtxt}}
  {}

\numberwithin{equation}{section}
\numberwithin{figure}{section}
\numberwithin{table}{section}
\usepackage{enumitem}		% customizable list environments
\theoremstyle{plain}
\newtheorem{thm}{\protect\theoremname}[section]
  \theoremstyle{definition}
  \newtheorem{defn}[thm]{\protect\definitionname}
  \theoremstyle{remark}
  \newtheorem{rem}[thm]{\protect\remarkname}
  \theoremstyle{plain}
  \newtheorem{lem}[thm]{\protect\lemmaname}
  \theoremstyle{plain}
  \newtheorem{prop}[thm]{\protect\propositionname}
  \theoremstyle{plain}
  \newtheorem{cor}[thm]{\protect\corollaryname}
  \theoremstyle{remark}
  \newtheorem*{claim*}{\protect\claimname}
  \theoremstyle{definition}
  \newtheorem{example}[thm]{\protect\examplename}
  \theoremstyle{remark}
  \newtheorem*{acknowledgement*}{\protect\acknowledgementname}

\providecommand{\MR}[1]{}

\usepackage{needspace}

\newref{lem}{refcmd={Lemma \ref{#1}}}
\newref{thm}{refcmd={Theorem \ref{#1}}}
\newref{cor}{refcmd={Corollary \ref{#1}}}
\newref{sec}{refcmd={Section \ref{#1}}}
\newref{sub}{refcmd={Section \ref{#1}}}
\newref{chap}{refcmd={Chapter \ref{#1}}}
\newref{prop}{refcmd={Proposition \ref{#1}}}
\newref{exa}{refcmd={Example \ref{#1}}}
\newref{tab}{refcmd={Table \ref{#1}}}
\newref{rem}{refcmd={Remark \ref{#1}}}
\newref{def}{refcmd={Definition \ref{#1}}}

\AtBeginDocument{
  
}

\makeatother

  \providecommand{\acknowledgementname}{Acknowledgement}
  \providecommand{\claimname}{Claim}
  \providecommand{\corollaryname}{Corollary}
  \providecommand{\definitionname}{Definition}
  \providecommand{\examplename}{Example}
  \providecommand{\lemmaname}{Lemma}
  \providecommand{\propositionname}{Proposition}
  \providecommand{\remarkname}{Remark}
\providecommand{\theoremname}{Theorem}

\begin{document}

\title{Graph Laplacians and discrete reproducing kernel Hilbert spaces from
restrictions}

\author{Palle Jorgensen and Feng Tian}

\address{(Palle E.T. Jorgensen) Department of Mathematics, The University
of Iowa, Iowa City, IA 52242-1419, U.S.A. }

\email{palle-jorgensen@uiowa.edu}

\urladdr{http://www.math.uiowa.edu/\textasciitilde{}jorgen/}

\address{(Feng Tian) Department of Mathematics, Trine University, Angola,
IN 46703, U.S.A.}

\email{tianf@trine.du}

\subjclass[2000]{Primary 47L60, 46N30, 65R10, 58J65, 81S25.}

\keywords{Reproducing kernel Hilbert space, discrete analysis, graph Laplacians,
distribution of point-masses, Green's functions. }

\maketitle
\pagestyle{myheadings}
\markright{}
\begin{abstract}
We study kernel functions, and associated reproducing kernel Hilbert
spaces $\mathscr{H}$ over infinite, discrete and countable sets $V$.
Numerical analysis builds discrete models (e.g., finite element) for
the purpose of finding approximate solutions to boundary value problems;
using multiresolution-subdivision schemes in continuous domains. In
this paper, we turn the tables: our object of study is realistic infinite
discrete models in their own right; and we then use an analysis of
suitable continuous counterpart problems, but now serving as a tool
for obtaining solutions in the discrete world.

\end{abstract}

\tableofcontents{}

\section{Introduction}

In a number of recent papers, kernel tools have found new applications,
and a number of them depend on an interplay between continuous vs
discrete; so between, (i) more classical continuous kernel models,
and (ii) various discretization procedures; see the citations below.
Applications of kernel tools include optimization, maximum-likelihood
constructs, and machine learning models; they all entail a combination
of analysis tools for a variety of reproducing kernels, and the associated
reproducing kernel Hilbert spaces (RKHS); as well as probabilistic
sampling and estimation; -- all issues involving theorems for RKHSs.
So far the emphasis has been on the continuous models, and we shall
turn the table in the present paper. We shall make use of a class
of discrete RKHSs which are typically associated with Gaussian free
fields, and determinantal point process, and associated determinantal
measures. The purpose of our paper is to offer a systematic approach
to these questions. Our approach is motivated in part by the way the
classical Cameron-Martin RKHS is used in the analysis of Brownian
motion, and related Gaussian processes. We are concerned with a characterization
of those RKHSs $\mathscr{H}$ of functions, on some state space $V$,
which contain the Dirac masses $\delta_{x}$ for all points $x$ in
$V$.

Our setting is that of infinite discrete models vs their continuous
counterparts. Our discrete analysis setting is as follows: Given is
an infinite set $V$ of vertices, and a set $E$ of edges, contained
in $V\times V\backslash\left\{ \mbox{the diagonal}\right\} $. Further,
a positive symmetric function $c$ on $E$ is prescribed ($c$ is
conductance in electrical networks), and there is a resulting graph
Laplacian. Its spectral theory will be considered. In this setting,
we arrive at a host of network models, whose analysis all involve
kernels. Connectedness for our infinite graphs will be assumed. The
relevant RKHSs are certain discrete Dirichlet spaces $\mathscr{H}_{\left(E,c\right)}$
consisting of finite energy-functions, and modeled on the classical
Cameron-Martin spaces; these discrete variants are Hilbert spaces
of functions on $V$ (modulo constants), and depending on the choice
of $\left(V,E,c\right)$. It is always the case that the differences
$f\left(x\right)-f\left(y\right)$ for $f$ in $\mathscr{H}_{\left(E,c\right)}$,
for any given pair of vertices $x$ and $y$, is well behaved: The
differences (voltage-drop) will be represented by a kernel (in $\mathscr{H}_{\left(E,c\right)}$),
depending on the pair $x$, $y$. By contrast, point-evaluation itself
(for a single vertex) does not automatically have a kernel representation
by a function in $\mathscr{H}_{\left(E,c\right)}$.

Our first result (\thmref{del}) gives a necessary and sufficient
condition for existence of $\mathscr{H}_{\left(E,c\right)}$-kernels
for point-evaluation, so a kernel associated to a fixed vertex $x_{0}$.
It follows from this that, when a vertex $x_{0}$ is given, the answer
to this question involves the entire vertex-set $V$, so including
distant vertices and also boundary considerations. The question of
deciding existence of finite-energy point-kernels has applications.
For example, each setting sketched above gives rise to an associated
Markov chain model involving a reversible random walk. It is known
that the random walk is transient if and only if finite-energy point-kernels
exist.

In Sections \ref{sec:dRKHS}-\ref{sec:drkhsv}, we turn to a family
of generalized Cameron-Martin spaces; -- in Theorems \ref{thm:main}
and \ref{thm:dk7}, we give an explicit comparison between the continuous
vs the discrete variants.

A \emph{reproducing kernel Hilbert space} (RKHS) is a Hilbert space
$\mathscr{H}$ of functions on a prescribed set, say $V$, with the
property that point-evaluation for $f\in\mathscr{H}$ is continuous
with respect to the $\mathscr{H}$-norm. They are called kernel spaces,
because, for every $x\in V$, the point-evaluation for functions $f\in\mathscr{H}$,
$f\left(x\right)$ must then be given as a $\mathscr{H}$-inner product
of $f$ and a vector $k_{x}$, in $\mathscr{H}$; called the kernel.
See (\ref{eq:pd31}) below.

\textbf{Background.} RKHSs have been studied extensively since the
pioneering papers by Aronszajn in the 1940ties, see e.g., \cite{Aro43,Aro48}.
They further play an important role in the theory of partial differential
operators (PDO); for example as Green's functions of second order
elliptic PDOs; see e.g., \cite{Nel57,HKL14,MR2301309}. Other applications
include engineering, physics, machine-learning theory (see \cite{KH11,SZ09,CS02}),
stochastic processes (e.g., Gaussian free fields), numerical analysis,
and more. See, e.g., \cite{AD93,ABDdS93,AD92,AJSV13,AJV14,BTA04}.
Also, see \cite{MR2089140,MR2607639,MR2913695,MR2975345,MR3091062,MR3101840,MR3201917}.
But the literature so far has focused on the theory of kernel functions
defined on continuous domains, either domains in Euclidean space,
or complex domains in one or more variables. For these cases, the
Dirac $\delta_{x}$ distributions do not have finite $\mathscr{H}$-norm.
But for RKHSs over discrete point distributions, it is reasonable
to expect that the Dirac $\delta_{x}$ functions will in fact have
finite $\mathscr{H}$-norm.

Here we consider the discrete case, i.e., RKHSs of functions defined
on a prescribed countable infinite discrete set $V$. We are concerned
with a characterization of those RKHSs $\mathscr{H}$ which contain
the Dirac masses $\delta_{x}$ for all points $x\in V$. Of the examples
and applications where this question plays an important role, we emphasize
two: (i) discrete Brownian motion-Hilbert spaces, i.e., discrete versions
of the Cameron-Martin Hilbert space \cite{Fer03,HH93}; (ii) energy-Hilbert
spaces corresponding to graph-Laplacians. 

Our setting is a given positive definite function $k$ on $V\times V$,
where $V$ is discrete (see above). We study the corresponding RKHS
$\mathscr{H}\left(=\mathscr{H}\left(k\right)\right)$ in detail. 

A positive definite kernel $k$ is said to be \emph{universal} \cite{CMPY08}
if, every continuous function, on a compact subset of the input space,
can be uniformly approximated by sections of the kernel, i.e., by
continuous functions in the RKHS. We show that for the RKHSs from
kernels $k_{c}$ in electrical network $G$ of resistors, this universality
holds. The metric in this case is the resistance metric on the vertices
of $G$, determined by the assignment of a conductance function $c$
on the edges in $G$, see \secref{drkhsv} below.

The problems addressed here are motivated in part by applications
to analysis on \emph{infinite weighted graphs}, to stochastic processes,
and to numerical analysis (discrete approximations), and to applications
of RKHSs to machine learning. Readers are referred to the following
papers, and the references cited there, for details regarding this:
\cite{MR3231624,MR2966130,MR2793121,MR3286496,MR3246982,MR2862151,MR3096457,MR3049934,MR2579912,MR741527,MR3024465}.

\textbf{Infinite Networks.} While the natural questions for the case
of large (or infinite) networks, \textquotedblleft the discrete world,\textquotedblright{}
have counterparts in the more classical context of partial differential
operators/equations (PDEs), the analysis on the discrete side is often
done without reference to a continuous PDE-counterpart.

The purpose of the present paper is to try to remedy this, to the
extent it is possible. We begin with the discrete context (\thmref{del}).
And we proceed to show that, in the discrete case, our analysis depends
on two tools, (i) positive definite (p.d.) functions, and associated
RKHSs, and (ii) (resistance) metrics. Both may be studied as purely
discrete objects, but nonetheless, in several of our results (including
the corollaries in Sections \ref{sec:net} and \ref{sec:drkhsv}),
we give contexts for continuous counterparts to the two discrete tools,
(i) and (ii). We make precise how to use the continuous counterparts
for computations in explicit discrete models, and in the associated
RKHSs.

In Theorems \ref{thm:main} and \ref{thm:main2} we give such concrete
(countable infinite) discrete models which can be understood as restrictions
of analogous PDE-models. In traditional numerical analysis, one builds
clever discrete models (finite element methods) for the purpose of
finding approximate solutions to PDE-boundary value problems. They
typically use multiresolution-subdivision schemes, applied to the
continuous domain, subdividing into simpler discretized parts, called
finite elements. And with variational methods, one then minimize various
error-functions. In this paper, we turn the tables: our object of
study are the discrete models, and analysis of suitable continuous
PDE boundary problems serve as a tool for solutions in the discrete
world.

\section{\label{sec:drkhs}Discreteness in reproducing kernel Hilbert spaces}
\begin{defn}
Let $V$ be a set, and $\mathscr{F}\left(V\right)$ denotes the set
of all \emph{finite} subsets of $V$. A function $k:V\times V\rightarrow\mathbb{C}$
is said to be \emph{positive definite}, if 
\begin{equation}
\underset{\left(x,y\right)\in F\times F}{\sum\sum}k\left(x,y\right)\overline{c_{x}}c_{y}\geq0\label{eq:pd1}
\end{equation}
holds for all coefficients $\{c_{x}\}_{x\in F}\subset\mathbb{C}$,
and all $F\in\mathscr{F}\left(V\right)$. 
\end{defn}

\begin{defn}
\label{def:d1}Fix a countable infinite set $V$.
\begin{enumerate}
\item For all $x\in V$, set 
\begin{equation}
k_{x}:=k\left(\cdot,x\right):V\rightarrow\mathbb{C}\label{eq:pd2}
\end{equation}
as a function on $V$. 
\item Let $\mathscr{H}:=\mathscr{H}\left(k\right)$ be the Hilbert-completion
of the $span\left\{ k_{x}:x\in V\right\} $, with respect to the inner
product 
\begin{equation}
\left\langle \sum_{x\in F}c_{x}k_{x},\sum_{y\in F}d_{y}k_{y}\right\rangle _{\mathscr{H}}:=\underset{F\times F}{\sum\sum}\overline{c_{x}}d_{y}k\left(x,y\right),\quad F\in\mathscr{F}\left(V\right),\label{eq:pd3}
\end{equation}
modulo the subspace of functions of zero $\mathscr{H}$-norm, i.e.,
\[
\underset{F\times F}{\sum\sum}\overline{c_{x}}c_{y}k\left(x,y\right)=0.
\]
$\mathscr{H}$ is then a reproducing kernel Hilbert space (RKHS),
with the reproducing property:
\begin{equation}
\left\langle k_{x},\varphi\right\rangle _{\mathscr{H}}=\varphi\left(x\right),\quad\forall x\in V,\:\forall\varphi\in\mathscr{H}.\label{eq:pd31}
\end{equation}

\item If $F\in\mathscr{F}\left(V\right)$, set $\mathscr{H}_{F}=\text{closed\:\ span}\{k_{x}\}_{x\in F}\subset\mathscr{H}$,
(closed is automatic as $F$ is finite.) And set 
\begin{equation}
P_{F}:=\text{the orthogonal projection onto \ensuremath{\mathscr{H}_{F}}}.\label{eq:pd4}
\end{equation}

\end{enumerate}
\end{defn}
\begin{rem}
The summations in (\ref{eq:pd3}) are all finite. We use physicists'
convention, so that the inner product in (\ref{eq:pd3}) is conjugate
linear in the first variable, and linear in the second variable.
\end{rem}
We shall need the following lemma:
\begin{lem}
\label{lem:s2-4}Let $k:V\times V\rightarrow\mathbb{C}$ be a positive
definite kernel, and let $\mathscr{H}$ be the corresponding RKHS.
Then a function $\xi$ on $V$ is in $\mathscr{H}$ if and only if
there is a constant $C<\infty$ such that, for all finite subsets
$F\subset V$, and all $\left\{ c_{x}\right\} _{x\in F}\subset\mathbb{C}^{\#F}$,
we have:
\begin{equation}
\left|\sum_{x\in F}c_{x}\xi\left(x\right)\right|^{2}\leq C\underset{F\times F}{\sum\sum}\overline{c_{x}}c_{y}k\left(x,y\right).
\end{equation}
\end{lem}
\begin{proof}
See \cite{Aro48}.
\end{proof}
A reproducing kernel Hilbert space (RKHS) is a Hilbert space of functions
on some set $S$. If $S$ comes with a topology, it is natural to
study RHHSs $\mathscr{H}$ consisting of continuous functions on $S$.
If $k$ is continuous on $S\times S$, one can show that then the
functions in $\mathscr{H}\left(k\right)$ are also continuous. Except
for trivial cases, the Dirac ``function'' 
\begin{equation}
\delta_{x}\left(y\right)=\begin{cases}
1 & y=x\\
0 & y\neq x
\end{cases}\label{eq:del0}
\end{equation}
is not continuous, and as a result, $\delta_{x}$ will typically \emph{not}
be in $\mathscr{H}$. But the situation is different for discrete
spaces.

We will show that, even if $S$ is a given discrete set (countable
infinite), we still often have $\delta_{x}\notin\mathscr{H}$ for
naturally arising RKHSs $\mathscr{H}$; see also \cite{JT15}. 

Below, we shall concentrate on cases when $S=V$ is a set of \emph{vertices}
in a \emph{graph} $G$ with \emph{edge-set} $E\subset V\times V\backslash\left(\text{diagonal}\right)$,
and on classes of RKHSs $\mathscr{H}$ of functions on $V$, for which
$\delta_{x}\in\mathscr{H}$ for all $x\in V$. 
\begin{defn}
\label{def:dmp}The RKHS $\mathscr{H}$ (in Def. \ref{def:d1}) is
said to have the \emph{discrete mass} property if $\delta_{x}\in\mathscr{H}$,
for all $x\in V$. $\mathscr{H}$ is then called a \emph{discrete
RKHS.}
\end{defn}
The following is immediate. 
\begin{prop}
\label{prop:dh5}Suppose $\mathscr{H}$ is a discrete RKHS of functions
on a vertex-set $V$, as described above; then there is a unique operator
$\Delta$, with dense domain $dom\left(\Delta\right)\subset\mathscr{H}$,
such that
\begin{equation}
\left(\Delta f\right)\left(x\right)=\left\langle \delta_{x},f\right\rangle _{\mathscr{H}},\quad\text{for all \ensuremath{x\in V}, and all \ensuremath{f\in dom\left(\Delta\right)}.}\label{eq:dh1}
\end{equation}

\end{prop}
The operator $\Delta$ from (\ref{eq:dh1}) will be studied in detail
in \secref{net} below. In special cases, it is called a \emph{graph-Laplacian};
and it is a discrete analogue of the classical Laplace operator $-\nabla^{2}$. 
\begin{cor}
Let $k:V\times V\rightarrow\mathbb{C}$, $\#V=\aleph_{0}$, be a positive
definite function, and let $\mathscr{H}$ be the corresponding RKHS.
Assume $\delta_{x}\in\mathscr{H}$ for all $x\in V$; then there are
closable operators
\begin{equation}
\mathscr{H}\xrightarrow{\;T\;}l^{2}\left(V\right),\quad\mbox{and}\quad l^{2}\left(V\right)\xrightarrow{\;S\;}\mathscr{H},\label{eq:aa1}
\end{equation}
such that $T\subseteq S^{*}$, $S\subseteq T^{*}$; and 
\begin{equation}
T\,k_{x}=\delta_{x},\;S\,\delta_{x}=\delta_{x},\;\mbox{for all \ensuremath{x\in V}.}\label{eq:aa2}
\end{equation}
\end{cor}
\begin{proof}
By assumption, $S$ and $T$ are well defined as specified in (\ref{eq:aa1})-(\ref{eq:aa2}),
with 
\begin{alignat*}{2}
dom\left(T\right) & =span\left\{ k_{x}\::\:x\in V\right\}  & \; & \mbox{dense in \ensuremath{\mathscr{H}}; and}\\
dom\left(S\right) & =span\left\{ \delta_{x}\::\:x\in V\right\}  & \; & \mbox{dense in \ensuremath{l^{2}\left(V\right)}.}
\end{alignat*}
By \propref{dh5}, we have
\begin{equation}
\left\langle Tf,\varphi\right\rangle _{l^{2}}=\left\langle f,S\varphi\right\rangle _{\mathscr{H}},\label{eq:aa3}
\end{equation}
for all $f\in dom\left(T\right)\subset\mathscr{H}$, and all $\varphi\in dom\left(S\right)\subset l^{2}$. 

The conclusions $T\subset S^{*}$, and $S\subset T^{*}$, follow from
(\ref{eq:aa3}). Since each operator $S$ and $T$ has dense domain,
it follows that both $T^{*}$ and $S^{*}$ must have dense domains
in the respective Hilbert spaces, i.e., $dom\left(T^{*}\right)$ dense
in $l^{2}$, and $dom\left(S^{*}\right)$ dense in $\mathscr{H}$. \end{proof}
\begin{rem}
As a result, we conclude that $T^{*}\overline{T}$ is a selfadjoint
extension of the operator $\Delta$ from \propref{dh5}.\end{rem}
\begin{thm}
\label{thm:del}Given $V$, and a positive definite (p.d.) function
$k:V\times V\rightarrow\mathbb{R}$, let $\mathscr{H}\left(=\mathscr{H}\left(k\right)\right)$
be the corresponding RKHS. Fix $x_{1}\in V$; then the following three
conditions are equivalent:
\begin{enumerate}
\item \label{enu:del1}$\delta_{x_{1}}\in\mathscr{H}$; 
\item \label{enu:del2}$\exists\,0\leq C_{x_{1}}<\infty$, such that 
\begin{equation}
\left|\xi\left(x_{1}\right)\right|^{2}\leq C_{x_{1}}\underset{F\times F}{\sum\sum}\overline{\xi\left(x\right)}\xi\left(y\right)k\left(x,y\right)\label{eq:d1}
\end{equation}
holds, for all $F\in\mathscr{F}\left(V\right)$, and all functions
$\xi$ on $F$. 
\item \label{enu:del3}For $F\in\mathscr{F}\left(V\right)$, set 
\begin{equation}
K_{F}:=\left(k\left(x,y\right)\right)_{\left(x,y\right)\in F\times F},\label{eq:d2}
\end{equation}
as a $\#F\times\#F$ matrix. Then
\begin{equation}
\sup_{F\in\mathscr{F}\left(V\right)}\left(K_{F}^{-1}\delta_{x_{1}}\right)\left(x_{1}\right)<\infty.\label{eq:d3}
\end{equation}

\end{enumerate}
\end{thm}
\begin{proof}
(\ref{enu:del1})\textbf{$\Longrightarrow$}(\ref{enu:del2}) We use
\lemref{s2-4}. Let $\xi$ be any function on $F\in\mathscr{F}\left(V\right)$,
and set 
\[
h_{\xi}:=\sum_{y\in F}\xi\left(y\right)k_{y}\left(\cdot\right)\in\mathscr{H}_{F};
\]
where $k_{y}\left(\cdot\right):=k\left(\cdot,y\right)$, as in (\ref{eq:pd2}).

Since $\delta_{x_{1}}\in\mathscr{H}$, we have 
\begin{eqnarray}
\left\langle \delta_{x_{1}},h_{\xi}\right\rangle _{\mathscr{H}} & = & \sum_{y\in F}\xi\left(y\right)\left\langle \delta_{x_{1}},k_{y}\right\rangle _{\mathscr{H}}\nonumber \\
 & \overset{\text{by \ensuremath{\left(\ref{eq:pd31}\right)}}}{=} & \sum_{y\in F}\xi\left(y\right)\delta_{x_{1}}\left(y\right)\nonumber \\
 & \overset{\text{by \text{\ensuremath{\left(\ref{eq:del0}\right)}}}}{=} & \xi\left(x_{1}\right).\label{eq:d30}
\end{eqnarray}
Moreover, the Cauchy-Schwarz inequality implies that 
\begin{eqnarray}
\left|\left\langle \delta_{x_{1}},h_{\xi}\right\rangle _{\mathscr{H}}\right|^{2} & \leq & \left\Vert \delta_{x_{1}}\right\Vert _{\mathscr{H}}^{2}\left\Vert h_{\xi}\right\Vert _{\mathscr{H}}^{2}\nonumber \\
 & = & \left\Vert \delta_{x_{1}}\right\Vert _{\mathscr{H}}^{2}\left\langle \sum_{x\in F}\xi\left(x\right)k_{x},\sum_{y\in F}\xi\left(y\right)k_{y}\right\rangle _{\mathscr{H}}\nonumber \\
 & \overset{\text{\mbox{by \ensuremath{\left(\ref{eq:pd3}\right)}}}}{=} & \left\Vert \delta_{x_{1}}\right\Vert _{\mathscr{H}}^{2}\underset{F\times F}{\sum\sum}\overline{\xi\left(x\right)}\xi\left(y\right)k\left(x,y\right).\label{eq:d31}
\end{eqnarray}
Therefore (\ref{eq:d1}) follows from (\ref{eq:d30}) and (\ref{eq:d31}),
with $C_{x_{1}}:=\left\Vert \delta_{x_{1}}\right\Vert _{\mathscr{H}}^{2}$.

(\ref{enu:del2})\textbf{$\Longrightarrow$}(\ref{enu:del3}) Recall
the matrix 
\[
K_{F}:=\left(\left\langle k_{x},k_{y}\right\rangle \right)_{\left(x,y\right)\in F\times F}
\]
as a linear operator $l^{2}\left(F\right)\rightarrow l^{2}\left(F\right)$,
where 
\begin{equation}
\left(K_{F}\varphi\right)\left(x\right)=\sum_{y\in F}K_{F}\left(x,y\right)\varphi\left(y\right),\;\varphi\in l^{2}\left(F\right).
\end{equation}
By (\ref{eq:d1}), we have 
\begin{equation}
\ker\left(K_{F}\right)\subset\left\{ \varphi\in l^{2}\left(F\right):\varphi\left(x_{1}\right)=0\right\} .
\end{equation}
Equivalently, 
\begin{equation}
\ker\left(K_{F}\right)\subset\left\{ \delta_{x_{1}}\right\} ^{\perp}
\end{equation}
and so $\delta_{x_{1}}\big|_{F}\in\ker\left(K_{F}\right)^{\perp}=\mbox{ran}\left(K_{F}\right)$,
and $\exists$ $\zeta^{\left(F\right)}\in l^{2}\left(F\right)$ s.t.
\begin{equation}
\delta_{x_{1}}\big|_{F}=\underset{=:h_{F}}{\underbrace{\sum\nolimits _{y\in F}\zeta^{\left(F\right)}\left(y\right)k\left(\cdot,y\right)}}.\label{eq:t0}
\end{equation}

\begin{claim*}
$P_{F}\left(\delta_{x_{1}}\right)=h_{F}$, where $P_{F}=$ projection
onto $\mathscr{H}_{F}$.\end{claim*}
\begin{proof}[Proof of the claim]
We only need to prove that $\delta_{x_{1}}-h_{F}\in\mathscr{H}\ominus\mathscr{H}_{F}$,
i.e., 
\begin{equation}
\left\langle \delta_{x_{1}}-h_{F},k_{z}\right\rangle _{\mathscr{H}}=0,\;\forall z\in F.\label{eq:t1}
\end{equation}
But, by (\ref{eq:t0}), 
\[
\text{LHS}_{\left(\ref{eq:t1}\right)}=\delta_{x_{1},z}-\sum_{y\in F}k\left(z,y\right)\zeta^{\left(F\right)}\left(y\right)=0.
\]

\end{proof}
Monotonicity: If $F\subset F'$, $F,F'\in\mathscr{F}\left(V\right)$,
then $\mathscr{H}_{F}\subset\mathscr{H}_{F'}$, and $P_{F}P_{F'}=P_{F}$
by easy facts for projections. Hence 
\begin{equation}
\left\Vert P_{F}\delta_{x_{1}}\right\Vert _{\mathscr{H}}^{2}\leq\left\Vert P_{F'}\delta_{x_{1}}\right\Vert _{\mathscr{H}}^{2},\quad h_{F}:=P_{F}\left(\delta_{x_{1}}\right)\label{eq:a0}
\end{equation}
and 
\begin{equation}
\lim_{F\nearrow V}\left\Vert \delta_{x_{1}}-h_{F}\right\Vert _{\mathscr{H}}=0.\label{eq:a1}
\end{equation}

Details: Since $\left\{ k_{x}\right\} _{x\in V}$ spans a dense subspace
in $\mathscr{H}$, by definition of $\mathscr{H}$, as a RKHS, we
conclude that 
\[
I_{\mathscr{H}}=\sup\left\{ P_{F}\:\big|\:F\in\mathscr{F}\left(V\right)\right\} 
\]
where $I_{\mathscr{H}}$ denotes the identity operator in $\mathscr{H}$.
We also use that the system of projections $\left\{ P_{F}\:|\:F\in\mathscr{F}\left(V\right)\right\} $
is a monotone filter in the following sense:

If $F,G\in\mathscr{F}\left(V\right)$, satisfying $F\subset G$, then
since $span\left\{ k_{x}\:|\:x\in F\right\} \subset span\left\{ k_{y}\:|\:y\in G\right\} $,
we get $P_{F}\mathscr{H}\subset P_{G}\mathscr{H}$, or equivalently
$P_{F}\leq P_{G}$, which is the same as $P_{F}=P_{F}P_{G}=P_{G}P_{F}$.
Hence 
\[
\left\Vert P_{G}h\right\Vert _{\mathscr{H}}^{2}=\left\Vert P_{F}h\right\Vert _{\mathscr{H}}^{2}+\left\Vert P_{L}h\right\Vert _{\mathscr{H}}^{2}\geq\left\Vert P_{F}h\right\Vert _{\mathscr{H}}^{2}
\]
holds for all $h\in\mathscr{H}$. The desired conclusion (\ref{eq:a1})
follows. 

Also note that 
\begin{align*}
\left\Vert P_{F}(\delta_{x_{1}})\right\Vert _{\mathscr{H}}^{2} & =\left\langle \delta_{x_{1}},P_{F}(\delta_{x_{1}})\right\rangle _{\mathscr{H}}\\
 & =K_{F}^{-1}\left(x_{1},x_{1}\right)=\left(K_{F}^{-1}\delta_{x_{1}}\right)\left(x_{1}\right)<\infty,
\end{align*}
which is (\ref{enu:del3}).

(\ref{enu:del3})\textbf{$\Longrightarrow$}(\ref{enu:del1}) The
argument from above shows that, for all $x\in V$, and all $\epsilon>0$,
there exists $F_{0}\in\mathscr{F}\left(V\right)$ such that, for all
$F,F'\in\mathscr{F}\left(V\right)$, $F\supseteq F_{0}$, $F'\supseteq F_{0}$,
we have 
\[
\left\Vert h_{F}-h_{F'}\right\Vert _{\mathscr{H}}<\epsilon,
\]
where (see (\ref{eq:t0})) 
\begin{align*}
\delta_{x_{1}}\big|_{F} & =K_{F}\zeta^{\left(F\right)}=h_{F}\big|_{F},\;\mbox{and}\\
\delta_{x_{1}}\big|_{F'} & =K_{F'}\zeta^{(F')}=h_{F'}\big|_{F'}.
\end{align*}

Passing to the limit in the filter (of finite subsets) $\mathscr{F}\left(V\right)$,
we now conclude that the limit $\delta_{x_{1}}\in\mathscr{H}$, which
is the desired conclusion (\ref{enu:del1}). For more details, we
refer to \cite{JT15}.
\end{proof}
If the condition in \propref{dh5} is satisfied we get an associated
operator $\Delta$ as specified in (\ref{eq:dh1}). But without additional
restrictions on $\mathscr{H}=\mathscr{H}\left(k,V\right)$ it is \emph{not}
automatic that $\Delta$ maps \emph{into} $\mathscr{H}$. 
\begin{thm}
\label{thm:s2-8}Let $k:V\times V\rightarrow\mathbb{C}$ be a positive
definite kernel, and let $\mathscr{H}=\mathscr{H}\left(k,V\right)$
be the corresponding RKHS. Assume that $\delta_{x}\in\mathscr{H}$
for all $x\in V$, so $\Delta$ (as in (\ref{eq:dh1})) is well defined. 

Then $\Delta k_{x}=\delta_{x}\left(\in\mathscr{H}\right)$ for all
$x\in V$, and so $\Delta$ is a densely defined Hermitian symmetric
operator in $\mathscr{H}$. \end{thm}
\begin{proof}
Let $x_{0}\in V$, then $\Delta k_{x_{0}}\in\mathscr{H}$ holds if
and only if there is an $\epsilon>0$ such that the following $x_{0}$-modified
kernel $k^{\left(x_{0}\right)}$ is positive definite, where 
\begin{equation}
k^{\left(x_{0}\right)}\left(x,y\right)=\begin{cases}
k\left(x,y\right) & \mbox{if }\left(x,y\right)\neq\left(x_{0},y_{0}\right)\\
k\left(x_{0},y_{0}\right)-\epsilon & \mbox{if }\left(x,y\right)=\left(x_{0},y_{0}\right).
\end{cases}\label{eq:kd1}
\end{equation}

We shall now show that this holds; it is an application of \thmref{del}.

Let $k,V,\mathscr{H}$ and $x_{0}\in V$ be as in the statement of
the theorem. Now consider the function $\xi=\Delta k_{x_{0}}$. 

\textbf{Step 1.} We use \lemref{s2-4}: To show that $\xi$ is in
$\mathscr{H}$, we shall verify a variant of the estimate (\ref{eq:d1})
from \thmref{del}. Consider all finite sums as follows, and the stated
\emph{a priori} estimate:
\begin{equation}
\Big|\sum_{y\in V}\lambda_{y}\xi\left(y\right)\Big|^{2}\leq C\underset{y,z}{\sum\sum}\overline{\lambda}_{y}\lambda_{z}k\left(y,z\right),\label{eq:kd2}
\end{equation}
$\forall\lambda_{y}\in\mathbb{C}$, finite support $y\in V$, where
$C<\infty$ depends only on $x_{0}$. We shall take $\epsilon=C^{-1}$
in (\ref{eq:kd1}).

\textbf{Step 2.} We have 
\begin{equation}
\xi\left(y\right)=\delta_{y,x_{0}},\quad y\in V.\label{eq:kd3}
\end{equation}
Verification of (\ref{eq:kd3}): We have
\[
\xi\left(y\right)=\left(\Delta k_{x_{0}}\right)\left(y\right)\underset{\text{by \ensuremath{\left(\ref{eq:dh1}\right)}}}{=}\left\langle \delta_{y},k_{x_{0}}\right\rangle =\delta_{x_{0},y}
\]
as claimed in (\ref{eq:kd3}). We used that $k_{x_{0}}$ has the reproducing
property, and that $\delta_{y}\in\mathscr{H}$ for all $y\in V$.

Hence we may apply \thmref{del}, (\ref{enu:del2})$\Longleftrightarrow$(\ref{enu:del1}),
to conclude that $\xi=\Delta k_{x_{0}}\in\mathscr{H}$. But, by Step
2, we also have $\xi=\delta_{x_{0}}\left(\in\mathscr{H}\right)$,
and so $\Delta_{k_{x_{0}}}=\delta_{x_{0}}$; which is the desired
conclusion.
\end{proof}

\section{Infinite networks}

In the section above we studied the discreteness property in the general
setting of reproducing kernel Hilbert spaces (RKHS). Below we turn
to our main application: Those RKHSs which arise in the study of infinite
network models; see e.g., \cite{CJ11,JP11,JT15}. By this we mean
infinite graphs $G=\left(V,E\right)$, with specified sets of vertices
$V$, and edges $E$ (see \defref{st0}). While such network models
$G$ have been previously studied in the literature, see e.g., \cite{JP10},
our present setting is more general in a number of respects; especially
in that our present setting, vertex points may have an infinite number
of neighbors, i.e., there may be points $x$ in $V$ with an infinite
number of edges $\left(xy\right)$.
\begin{rem}
Our present paper has 3 different settings of generality:
\begin{enumerate}
\item \label{enu:s2-1}The RHKSs in general;
\item \label{enu:s2-2}The special RKHSs $\mathscr{H}$ which has the \emph{discrete
mass property} (\defref{dmp}), i.e., containing all the Dirac masses. 
\item \label{enu:s2-3}The RKHSs from infinite network models $\left(V,E,c\right)$,
where $V$ consists of vertices, $E$ is the edge-set in $V\times V\backslash\left\{ \mbox{diagonal}\right\} $,
and $c$ is a prescribed conductance function on $E$. 
\end{enumerate}

(Note that (\ref{enu:s2-3}) is a special case of (\ref{enu:s2-2}),
and (\ref{enu:s2-2}) is a special case of (\ref{enu:s2-1}). See
details below.)\\

In general, if $k:V\times V\rightarrow\mathbb{C}$ is a p.d. function,
we get $\mathscr{H}=\mbox{RKHS}\left(k,V\right)$. 

In \thmref{s2-8}, we showed that \emph{if }$\mathscr{H}$ is special,
having the discrete mass property (\defref{dmp}), i.e., $\delta_{x}\in\mathscr{H}$,
for all $x\in V$; then we may consider the function 
\begin{equation}
\left(\Delta f\right)\left(x\right):=\left\langle \delta_{x},f\right\rangle _{\mathscr{H}}\label{eq:s2-s1}
\end{equation}
as in (\ref{eq:dh1}). But it is not guaranteed that the function
$V\ni x\longrightarrow\left(\Delta f\right)\left(x\right)$ will be
in $\mathscr{H}$. In the proof of \thmref{s2-8}, we showed, using
\lemref{s2-4}, that the operator $\Delta$ in (\ref{eq:s2-s1}) indeed
maps into $\mathscr{H}$. 

Consequently, setting $dom\left(\Delta\right)=span\left\{ k_{x}\::\:x\in V\right\} $,
$\Delta k_{x}=\delta_{x}$, it follows that
\[
\left\langle \Delta k_{x},k_{y}\right\rangle _{\mathscr{H}}=\left\langle k_{x},\Delta k_{y}\right\rangle _{\mathscr{H}}=\delta_{x,y},\quad x,y\in V;
\]
and so $\Delta$ is a densely defined Hermitian symmetric operator
in $\mathscr{H}$. 

Now, specialize to our infinite network models. Let $\left(V,E,c\right)$
be as above; pick a base-point $o\in V$. Set 
\begin{equation}
V':=V\backslash\left\{ o\right\} .
\end{equation}
By Riesz' theorem (see also \cite{JP10,JP11}), there exists $v_{x}:=v_{x,o}$,
such that 
\begin{equation}
f\left(x\right)-f\left(o\right)=\left\langle v_{x},f\right\rangle _{\mathscr{H}}\label{eq:s2-s3}
\end{equation}
valid for all $f\in\mathscr{H}$, and all $x\in V'$. We may further
assume that $v_{x}\left(o\right)=0$, $x\in V'$. The functions $v_{x}$
are called \emph{dipoles}. 

Consider the \emph{energy Hilbert space}, $\mathscr{H}_{E,c}$, with
inner product defined by 
\begin{equation}
\left\langle f,g\right\rangle _{\mathscr{H}_{E,c}}=\frac{1}{2}\underset{\left(s,t\right)\in E}{\sum\sum}c_{st}(\overline{f\left(s\right)}-\overline{f\left(t\right)})\left(g\left(s\right)-g\left(t\right)\right).\label{eq:s2-s4}
\end{equation}
Set 
\begin{align}
k^{\left(c\right)}\left(x,y\right) & :=\left\langle v_{x},v_{y}\right\rangle _{\mathscr{H}_{E}}=\frac{1}{2}\underset{\left(s,t\right)\in E}{\sum\sum}c_{st}\left(v_{x}\left(s\right)-v_{x}\left(t\right)\right)\left(v_{y}\left(s\right)-v_{y}\left(t\right)\right)\nonumber \\
 & =v_{x}\left(y\right),\quad\forall x,y\in V';\label{eq:s2-s5}
\end{align}
where we have used the property from (\ref{eq:s2-s3}).\end{rem}
\begin{defn}
\label{def:st0}Let $V$ be a set, $\#V=\aleph_{0}$, and let $E\subset V\times V\backslash\left\{ \mbox{diagonal}\right\} $.
Let $c:E\rightarrow[0,\infty)$ be a fixed function. We assume that
for any pair $x,y\in V$ $\exists$ $\left\{ x_{i}\right\} _{i=0}^{n}\subset V$
s.t. $\left(x_{i},x_{i+1}\right)\in E$, $x_{0}=x$ and $x_{n}=y$. 

For functions on $V$, we introduce the following equivalence relation
\[
f_{1}\sim f_{2}\underset{\text{Def}}{\Longleftrightarrow}f_{1}-f_{2}\;\mbox{is constant on \ensuremath{V}.}
\]
On the set of equivalence classes, we define the $c$-energy 
\[
\left\Vert f\right\Vert _{\mathscr{H}_{E,c}}^{2}=\frac{1}{2}\underset{\left(x,y\right)\in E}{\sum\sum}c_{xy}\left|f\left(x\right)-f\left(y\right)\right|^{2}<\infty;
\]
see (\ref{eq:s2-s4})-(\ref{eq:s2-s5}).
\end{defn}
With the corresponding inner product, this becomes a Hilbert space
denoted $\mathscr{H}_{E,c}$. 

Using this and Riesz, we showed that, for all $x,y\in V$, $\exists!$
$v_{xy}\in\mathscr{H}_{E,c}$ s.t. 
\begin{equation}
f\left(x\right)-f\left(y\right)=\left\langle v_{xy},f\right\rangle _{\mathscr{H}_{E,c}},\quad f\in\mathscr{H}_{E,c}.\label{eq:st1}
\end{equation}
Fix a base-point $o\in V$, and set 
\begin{equation}
k^{\left(c\right)}\left(x,y\right)=\left\langle v_{xo},v_{yo}\right\rangle _{\mathscr{H}_{E,c}}.\label{eq:st2}
\end{equation}
Then $k^{\left(c\right)}$ is a positive definite kernel, and we get
a canonical isomorphism
\[
\mbox{RKHS}(k^{\left(c\right)})\simeq\left\{ f\in\mathscr{H}_{E,c}\::\:f\left(o\right)=0\right\} .
\]
We normalize with $v_{x,o}\left(0\right)=0$. 
\begin{cor}
Let $V,E,c$ and $\mathscr{H}_{E,c}$ be as in \defref{st0}, then,
for $x\in V$, $\mbox{class}\left(\delta_{x}\right)\in\mathscr{H}_{E,c}$
holds, if and only if, 
\[
c\left(x\right):=\underset{\substack{y\in V\\
\left(x,y\right)\in E
}
}{\sum}c_{xy}<\infty.
\]
In this case, 
\[
\left\Vert \delta_{x}\right\Vert _{\mathscr{H}_{E,c}}^{2}=c\left(x\right).
\]
\end{cor}
\begin{proof}
The result is immediate from \thmref{del}; see especially (\ref{eq:d3}).
\end{proof}

\section{\label{sec:dRKHS}Discrete RKHSs as restrictions}

Given a discrete set $V$, $\#V=\aleph_{0}$, let $K_{V}:V\times V\rightarrow\mathbb{C}$
(or $\mathbb{R}$) be a positive definite (p.d.) function, and $\mathscr{H}_{V}=\mathscr{H}\left(K_{V}\right)$
the corresponding RKHS. We study when $\delta_{x}$ is in $\mathscr{H}_{V}$,
for all $x\in V$.

We show below (\thmref{main}; also see \secref{net}) that every
fundamental solution for a Dirichlet boundary value problem on a bounded
open domain $\Omega$ in $\mathbb{R}^{\nu}$, allows for discrete
restrictions (i.e., vertices sampled in $\Omega$), which have the
desired ``discrete mass'' property (see \defref{dmp}). 
\begin{rem}
To get the desired conclusions, consider a continuous p.d. function
$K:\Omega\times\Omega\rightarrow\mathbb{R}$, and $\left\{ f_{n}\right\} _{n\in\mathbb{N}}$
an ONB for $\mathscr{H}\left(K\right)$ = RKHS = CM (Cameron Martin
Hilbert space; see \subref{CM}.) We need suitable restricting assumptions
on the prescribed set $\Omega\subset\mathbb{R}^{\nu}$: 
\begin{enumerate}[label=(\roman{enumi})]
\item \label{enu:d1}open
\item bounded
\item connected
\item \label{enu:d4}smooth boundary $\partial\Omega$
\end{enumerate}

(Some of the restrictions on $\Omega$ may be relaxed, but even this
setting is interesting enough.)

The conditions \ref{enu:d1}-\ref{enu:d4} are satisfied by the covariance
function $k\left(s,t\right)=s\wedge t-st$ on $\Omega\times\Omega$,
$\Omega=\left(0,1\right)\subset\mathbb{R}$, for Brownian bridge (Examples
\ref{exa:1} and \ref{exa:bb}). By contrast, we also have the covariance
function $k\left(s,t\right)=s\wedge t$ for the Brownian motion, and
in this case, we may take $\Omega=\mathbb{R}_{+}$, or $\Omega=\mathbb{R}$;
and these are examples with \emph{unbounded} domains $\Omega\subset\mathbb{R}^{\nu}$. 
\end{rem}

\subsection{Examples from elliptic operators}

Given a bounded open domain $\Omega\subset\mathbb{R}^{\nu}$, let
\begin{equation}
\Delta=-\sum_{j=1}^{\nu}\left(\frac{\partial}{\partial x_{j}}\right)^{2}=-\nabla^{2}\label{eq:e1}
\end{equation}
with corresponding Green's function $K$, i.e., the fundamental solution
to the Dirichlet boundary value problem, so that $K:\Omega\times\Omega\rightarrow\mathbb{C}$,
$\Delta K\left(s,\cdot\right)=\delta_{s}$, and $K\left(s,\cdot\right)\big|{}_{\partial\Omega}\equiv0$. 

We assume that $\Omega\subset\mathbb{R}^{\nu}$ has finite Poincaré
constant $C_{P}$, i.e., we have: 
\begin{equation}
\int_{\Omega}\left|f\right|^{2}dx\leq C_{P}\left\{ \left|\int_{\Omega}f\,dx\right|^{2}+\int_{\Omega}\left|\nabla f\right|^{2}dx\right\} ,\quad\forall f\in\mathscr{H}_{CM}\left(\Omega\right);\label{eq:pc}
\end{equation}
see \cite{Ami78,DL54}.
\begin{example}[Brownian bridge]
\label{exa:1}For $\nu=1$, $\Omega=\left(0,1\right)$, then 
\begin{equation}
k\left(s,t\right)=s\wedge t-st;\label{eq:e2}
\end{equation}
is the covariance function of the Brownian bridge.\end{example}
\begin{prop}
Let $\Omega=\left(0,1\right)$, and $k:\Omega\times\Omega\rightarrow\mathbb{R}$
be the covariance function in (\ref{eq:e2}). Set $k_{s}\left(t\right):=k\left(s,t\right)$,
for all $s,t\in\Omega$. Then, the function $k_{s}$ satisfies:
\begin{equation}
\begin{split}k_{s}\left(0\right) & =k_{s}\left(1\right)=0,\quad\mbox{and}\quad-\left(\frac{d}{dt}\right)^{2}k_{s}=\delta\left(s-t\right).\end{split}
\label{eq:e3}
\end{equation}
\end{prop}
\begin{proof}
Direct verification, see \cite{JT15}. Sketch: Fix $s$, $0<s<1$,
then 
\[
k_{s}\left(t\right)=\begin{cases}
t\left(1-s\right) & 0<t\leq s\\
s\left(1-t\right) & s<t<1
\end{cases}
\]
and (see Fig \ref{fig:bbf}) 
\begin{eqnarray*}
\frac{d}{dt}k_{s}\left(t\right) & = & \begin{cases}
1-s & 0<t\leq s\\
-s & s<t<1
\end{cases}\\
 & = & \left(1-s\right)\chi_{\left(0,s\right)}\left(t\right)-s\chi_{\left(s,1\right)}\left(t\right)\\
\left(\frac{d}{dt}\right)^{2}k_{s}\left(t\right) & = & -\left(1-s\right)\delta_{s}-s\delta_{s}=-\delta_{s};
\end{eqnarray*}
or equivalently, $\Delta K=\delta\left(t-s\right)$, where $k$ is
as in (\ref{eq:e2}).\end{proof}
\begin{example}[Discrete version]
Fix $R$ s.t. $0<R<1$, set $V=\mathbb{Z}_{+}$ and conductance $c_{i,i+1}=R^{-i}$.
Then the energy-Hilbert space $\mathscr{H}_{R}$ consists of functions
$f$ on $\mathbb{Z}_{+}$ s.t.
\begin{align}
 & \lim_{i\rightarrow\infty}f\left(i\right)=0,\;\mbox{and}\label{eq:h1}\\
 & \left\Vert f\right\Vert _{\mathscr{H}_{R}}^{2}=\sum_{i=1}^{\infty}\frac{1}{R^{i-1}}\left|f\left(i\right)-f\left(i-1\right)\right|^{2}<\infty,\label{eq:h2}
\end{align}
and graph-Laplacian
\begin{equation}
\left(\Delta_{disc}f\right)\left(i\right)=R^{-i+1}\left(f\left(i\right)-f\left(i-1\right)\right)+R^{-i}\left(f\left(i\right)-f\left(i+1\right)\right).\label{eq:h4}
\end{equation}
In this case, the reproducing kernel $k_{R}$ is as follows:
\begin{equation}
k_{R}\left(i,j\right)=\left\langle v_{i}^{R},v_{j}^{R}\right\rangle _{\mathscr{H}_{R}}=\frac{R^{i\wedge j}}{1-R},\quad\forall\left(i,j\right)\in\mathbb{Z}_{+}\times\mathbb{Z}_{+}.\label{eq:h3}
\end{equation}
Here $\left(v_{i}^{R}\right)_{i\in\mathbb{Z}_{+}}$ is a system of
functions in $\mathscr{H}_{R}$ such that $\left\langle v_{i}^{R},f\right\rangle _{\mathscr{H}_{R}}=f\left(i\right)$,
$\forall f\in\mathscr{H}_{R}$.
\end{example}

\subsubsection*{Higher dimensions}

Let $\Omega\subset\mathbb{R}^{\nu}$, bounded and open s.t. $\partial\Omega$
is smooth. Set 
\begin{equation}
\begin{split} & \Delta_{0}:=-\sum_{j=1}^{\nu}\left(\frac{\partial}{\partial x_{j}}\right)^{2},\;\text{(see \ensuremath{\left(\ref{eq:e1}\right)}) with}\\
 & dom\left(\Delta_{0}\right)=\left\{ f\in L^{2}\left(\Omega\right)\:\big|\:\Delta f\in L^{2}\left(\Omega\right),\;\mbox{and }f\big|_{\partial\Omega}\equiv0\right\} .
\end{split}
\label{eq:e4}
\end{equation}
We have that $\Delta_{0}$ is selfadjoint in $L^{2}\left(\Omega\right)$
and that 
\begin{equation}
\Delta_{0}\geq0\;\text{on \ensuremath{dom(\Delta_{0})}};\label{eq:e5}
\end{equation}
i.e., $\left\langle \varphi,\Delta_{0}\varphi\right\rangle _{L^{2}\left(\Omega\right)}\geq0$,
$\forall\varphi\in dom(\Delta_{0})$; and we therefore get its kernel
$K$ (analogous to (\ref{eq:e3}) in Example \ref{exa:1} (Brownian
bridge).) By (\ref{eq:e4})-(\ref{eq:e5}), 
\[
S_{t}=e^{-t\Delta_{0}}:L^{2}\left(\Omega\right)\rightarrow L^{2}\left(\Omega\right),\quad t>0
\]
is a s.a. contractive semigroup. 

Let 
\begin{equation}
K=\int_{0}^{\infty}e^{-t\Delta_{0}}dt:L^{2}\left(\Omega\right)\rightarrow L^{2}\left(\Omega\right)\label{eq:e6}
\end{equation}
as an operator, generally unbounded. Since $\Delta_{0}$ is elliptic,
we further get that $K$ is represented as 
\begin{equation}
\left(Kf\right)\left(x\right)=\int_{\Omega}K\left(x,y\right)f\left(y\right)dy\label{eq:e7}
\end{equation}
where the integral in (\ref{eq:e7}) is w.r.t. Lebesgue measure in
$\mathbb{R}^{\nu}$. 
\begin{lem}
\label{lem:dRKHS5}Let $K$ be the kernel in (\ref{eq:e6}), then
\begin{equation}
\Delta_{0}K=I\;\text{on \ensuremath{L^{2}\left(\Omega\right)}.}\label{eq:e9}
\end{equation}
Moreover, $K:\Omega\times\Omega$ is continuous, and p.d., i.e., 
\begin{equation}
\sum_{i}\sum_{j}\xi_{i}\xi_{j}K\left(x_{i},x_{j}\right)\geq0\label{eq:e8}
\end{equation}
for all coefficients $\left\{ \xi_{i}\right\} _{i=1}^{n}\subset\mathbb{R}$,
and all $\left\{ x_{i}\right\} _{i=1}^{n}\subset\Omega$. \end{lem}
\begin{proof}
Let $P_{2}\left(d\lambda\right)$ denote the spectral resolution of
$\Delta_{0}$ (projection valued measure (PVM)), i.e., 
\begin{equation}
\begin{split} & P_{2}\left(A\right)=P_{2}\left(A\right)^{*}=P_{2}\left(A\right)^{2}\\
 & P_{2}\left(A\cap B\right)=P_{2}\left(A\right)P_{2}\left(B\right),\quad\forall A,B\in\mathscr{B}\left(\mathbb{R}_{+}\right),
\end{split}
\label{eq:e11}
\end{equation}
so that 
\begin{equation}
\begin{split}\Delta_{0} & =\int_{0}^{\infty}\lambda P_{2}\left(d\lambda\right)\\
e^{-t\Delta_{0}} & =\int_{0}^{\infty}e^{-t\lambda}dP_{2}\left(\lambda\right),\quad t>0
\end{split}
\label{eq:e12}
\end{equation}
and 
\begin{equation}
K=\int_{0}^{\infty}\frac{1}{\lambda}P_{2}\left(d\lambda\right)=\Delta_{0}^{-1}\label{eq:e13}
\end{equation}
(in general an unbounded operator.) 

To see that $K$ is p.d. (see (\ref{eq:e8})):

Step 1. If $\varphi\in C_{c}^{\infty}\left(\Omega\right)$; then (enough
to consider the real valued case)
\begin{eqnarray*}
 &  & \int_{\Omega}\int_{\Omega}K\left(x,y\right)\varphi\left(x\right)\varphi\left(y\right)dxdy\\
 & = & \left\langle \varphi,K\varphi\right\rangle _{2}=\int_{0}^{\infty}\frac{1}{\lambda}\left\Vert P_{2}\left(d\lambda\right)\varphi\right\Vert _{2}^{2}\geq0.
\end{eqnarray*}

Step 2. Approximate $\left(\delta_{x}\right)$ with $C_{c}^{\infty}\left(\Omega\right)$. \end{proof}
\begin{cor}
\label{cor:dRKHS6}Let $\Delta_{0}$, and $\Omega$ be as in \lemref{dRKHS5}.
Then we have a system of kernels $p_{t}\left(x,y\right)$, $t\geq0$,
$\left(x,y\right)\in\Omega\times\Omega$, such that
\begin{enumerate}
\item $p_{0}\left(x,y\right)=\delta_{x,y}$;
\item $p_{t}\left(x,\cdot\right)\in L^{2}\left(\Omega\right)$ for all $t\in\mathbb{R}_{+}$,
$x\in\Omega$; 
\item each $p_{t}\left(\cdot,\cdot\right)$ is a positive definite kernel
on $\Omega\times\Omega$; 
\item \label{enu:dr4}for $s,t\in[0,\infty)$, we have 
\[
p_{s+t}\left(x,y\right)=\int_{\Omega}p_{s}\left(x,z\right)p_{t}\left(z,y\right)\,dz;\;\mbox{and}
\]

\item the kernel $K$ from (\ref{eq:e13}) satisfies 
\[
K\left(x,y\right)=\int_{0}^{\infty}p_{t}\left(x,y\right)dt,\quad\forall\left(x,y\right)\in\Omega\times\Omega.
\]

\end{enumerate}
\end{cor}
\begin{proof}
Immediate from the lemma and an application of the Spectral Theorem. \end{proof}
\begin{rem}
The conclusions in the corollary are also valid in our discrete models,
$\left(V,E,c\right)$, $E\subset V\times V\backslash\left\{ \mbox{diagonal}\right\} $,
and $c$ defined on $E$ as in \defref{g} below. In this case, 
\begin{equation}
\left(\Delta_{0}f\right)\left(x\right)=\sum_{\substack{y\\
\left(x,y\right)\in E
}
}c_{xy}\left(f\left(x\right)-f\left(y\right)\right),\quad f\in l^{2}\left(V\right);
\end{equation}
\begin{equation}
e^{-t\Delta_{0}}:l^{2}\left(V\right)\longrightarrow l^{2}\left(V\right),\quad t\in\mathbb{R}_{+};
\end{equation}
\begin{equation}
\left(e^{-t\Delta_{0}}f\right)\left(x\right)=\sum_{y\in V}p_{t}\left(x,y\right)f\left(y\right),\quad\forall f\in l^{2}\left(V\right),\:t\in\mathbb{R}_{+};
\end{equation}
and (analogous to (\ref{enu:dr4}) in the corollary), we have:
\begin{equation}
p_{s+t}\left(x,y\right)=\sum_{z\in V}p_{s}\left(x,z\right)p_{t}\left(z,y\right),\quad\forall\left(x,y\right)\in V\times V,\:\forall s,t\in[0,\infty).
\end{equation}
\end{rem}
\begin{cor}
\label{cor:K}Let $K$ be as in (\ref{eq:e6}), then $K$ satisfies
that $\Delta K=\delta\left(x-y\right)$ on $\Omega\times\Omega$,
and $K\left(x,\cdot\right)\big|{}_{\partial\Omega}\equiv0$.\end{cor}
\begin{proof}
By well-known facts from elliptic operators, the conclusion is equivalent
to (\ref{eq:e9}).\end{proof}
\begin{lem}
\label{lem:sp}Let $k:\Omega\times\Omega\rightarrow\mathbb{C}$ be
a p.d. kernel, and let $\mathscr{H}=\mathscr{H}\left(k\right)$ be
the corresponding RKHS. Then for every subset $V\subset\Omega$, $\mathscr{H}\big|_{V}=\left\{ f\big|_{V}\:;\:f\in\mathscr{H}\right\} $
is a RKHS $\mathscr{H}_{V}$; and if $\varphi\in\mathscr{H}_{V}$,
then $\left\Vert \varphi\right\Vert _{\mathscr{H}_{V}}=\inf\left\{ \left\Vert f\right\Vert _{\mathscr{H}}\::\:f\in\mathscr{H},\;f\big|_{V}=\varphi\right\} $. \end{lem}
\begin{proof}
See \cite{Aro43,Aro48}. Consider $\mathscr{H}\ominus\mathscr{N}_{V}$
(closed), where $\mathscr{N}_{V}=\left\{ f\in\mathscr{H}\:;\:f\big|_{V}=0\right\} $.
Moreover the $\inf$ is attained; $f\in\mathscr{H}$.\end{proof}
\begin{cor}
Let $k$, $V$, and $\mathscr{H}=\mathscr{H}\left(k\right)$ be as
above. For $f\in\mathscr{H}$, write 
\begin{equation}
f=f_{0}+f_{1}\label{eq:n1}
\end{equation}
w.r.t. the orthogonal splitting (two closed subspaces):
\begin{equation}
\mathscr{H}=\mathscr{N}_{V}\oplus\mathscr{N}_{V}^{\perp},\;f_{0}\in\mathscr{N}_{V},\:f_{1}\in\mathscr{N}_{V}^{\perp},\label{eq:n2}
\end{equation}
then if $\varphi\in\mathscr{H}_{V}$, we have 
\begin{equation}
\left\Vert \varphi\right\Vert _{\mathscr{H}_{V}}=\left\Vert f_{1}\right\Vert _{\mathscr{H}},\label{eq:n3}
\end{equation}
where $\varphi=f\big|_{V}$. \end{cor}
\begin{proof}
With the splitting (\ref{eq:n1}), we get $f\big|_{V}=f_{1}\big|_{V}$;
so 
\begin{equation}
\varphi=f\big|_{V}=f_{1}\big|_{V},\;\mbox{and}\label{eq:n4}
\end{equation}
\[
\left\Vert f\right\Vert _{\mathscr{H}}^{2}=\left\Vert f_{0}\right\Vert _{\mathscr{H}}^{2}+\left\Vert f_{1}\right\Vert _{\mathscr{H}}^{2}\geq\left\Vert f_{1}\right\Vert _{\mathscr{H}}^{2}
\]
so by (\ref{eq:n4}), $\left\Vert \varphi\right\Vert _{\mathscr{H}_{V}}=\left\Vert f_{1}\right\Vert _{\mathscr{H}}$. 
\end{proof}
The purpose of the next section is to study these restrictions (discrete)
in detail, from cases where $\mathscr{H}$ is one of the classical
\emph{continuous} RKHSs.

\subsection{\label{sub:CM}The Cameron-Martin space $\mathscr{H}_{CM}\left(\Omega\right)$}

The Cameron\textendash Martin Hilbert space is a RKHS (abbreviated
C-M below) which gives the context for the Cameron\textendash Martin
formula which describes how abstract Wiener measure changes under
translation by elements of the Cameron-Martin RKHS. Context: Abstract
Wiener measure is quasi-invariant (under translation), not invariant;
and the C-M RKHS serves as a tool in a formula for computing of the
corresponding Radon-Nikodym derivatives, the C-H formula; see e.g.,
\cite{HH93}. The technical details involved vary, depending on the
dimension, and on suitable boundary conditions, see below.

Let $\Omega\subset\mathbb{R}^{\nu}$, satisfying conditions \ref{enu:d1}-\ref{enu:d4};
i.e., $\Omega$ is bounded, open, and connected in $\mathbb{R}^{\nu}$
with smooth boundary $\partial\Omega$. 

Let $K:\Omega\times\Omega\rightarrow\mathbb{R}$ continuous, p.d.,
given as the Green's function of $\Delta_{0}$, for the Dirichlet
boundary condition, see (\ref{eq:e4}). Thus, $\Delta_{0}$ is positive
selfadjoint, and 
\begin{align}
 & \Delta K=\delta\left(x-y\right)\text{ on \ensuremath{\Omega\times\Omega}}\label{eq:m1}\\
 & K\left(x,\cdot\right)\big|_{\partial\Omega}\equiv0\label{eq:m2}
\end{align}
see Corollary \ref{cor:K}.

Let $\mathscr{H}_{CM}\left(\Omega\right)$ be the corresponding Cameron-Martin
RKHS. 

For $\nu=1$, $\Omega=\left(0,1\right)$, take
\begin{equation}
\begin{split}\mathscr{H}_{CM}\left(0,1\right)= & \Big\{ f\:\big|\:f'\in L^{2}\left(0,1\right),\;f\left(0\right)=f\left(1\right)=0,\\
 & \left\Vert f\right\Vert _{CM}^{2}:=\int_{0}^{1}\left|f'\right|^{2}dx<\infty\Big\}
\end{split}
\label{eq:m3}
\end{equation}

For $\nu>1$, let
\begin{equation}
\begin{split}\mathscr{H}_{CM}\left(\Omega\right)= & \left\{ f\:\big|\:\nabla f\in L^{2}\left(\Omega\right),\:f\big|_{\partial\Omega}\equiv0,\:\left\Vert f\right\Vert _{CM}^{2}:=\int_{\Omega}\left|\nabla f\right|^{2}dx<\infty\right\} ,\\
 & \text{ where }\nabla=\left(\frac{\partial}{\partial x_{1}},\frac{\partial}{\partial x_{2}},\cdots,\frac{\partial}{\partial x_{\nu}}\right).
\end{split}
\label{eq:m4}
\end{equation}

\begin{rem}
In the case of $\Omega=\left(0,1\right)$, $\nu=1$, and for $K\left(s,t\right)=s\wedge t-st$,
we have $\mathscr{H}_{CM}\left(0,1\right)$ as in (\ref{eq:m3}).
The following decomposition holds:
\[
K\left(s,t\right)=\sum_{n=1}^{\infty}\frac{\sin\left(n\pi s\right)\sin\left(n\pi t\right)}{\left(n\pi\right)^{2}},\quad\left(s,t\right)\in\Omega\times\Omega.
\]
\end{rem}
\begin{proof}
Use Fourier series; or the fact that 
\[
\left\{ \frac{\sin\left(n\pi t\right)}{n\pi}\right\} _{n=1}^{\infty}
\]
is an ONB in $\mathscr{H}_{CM}\left(0,1\right)$. 
\end{proof}
In general, $\nu>1$, there exists ONB $\left\{ f_{n}\right\} _{n\in\mathbb{N}}$
in $\mathscr{H}_{CM}\left(\Omega\right)$ (see (\ref{eq:m4})), such
that 
\[
K\left(x,y\right)=\sum_{n=1}^{\infty}f_{n}\left(x\right)f_{n}\left(y\right)\;\text{on \ensuremath{\Omega\times\Omega}.}
\]

\begin{proof}
A result from the theory of RKHS. \end{proof}
\begin{lem}[The reproducing property]
 Let $K$ be the kernel of $\Delta_{0}$ for the Dirichelet boundary
condition; and let $\mathscr{H}_{CM}\left(\Omega\right)$ be the Cameron-Martin
space in (\ref{eq:m4}). Then 
\begin{equation}
\left\langle K_{x},f\right\rangle _{CM}=f\left(x\right),\quad\forall f\in\mathscr{H}_{CM},\:\forall x\in\Omega.\label{eq:m5}
\end{equation}
\end{lem}
\begin{proof}
Note that 
\begin{eqnarray*}
\mbox{LHS}{}_{\text{\ensuremath{\left(\ref{eq:m5}\right)}}} & = & \int_{\Omega}\left(\nabla K_{x}\right)\left(\nabla f\right)dy\\
 & = & -\int_{\Omega}\left(\nabla^{2}K_{x}\right)f\left(y\right)dy\quad\left(\text{by \text{\text{\ensuremath{\left(\ref{eq:m2}\right)},\ensuremath{\left(\ref{eq:m4}\right)}}}}\right)\\
 & \underset{\text{(by \ensuremath{\left(\ref{eq:m1}\right)})}}{=} & \int_{\Omega}\Delta K_{x}f\left(y\right)dy\\
 & \underset{\text{(by \ensuremath{\left(\ref{eq:m1}\right)})}}{=} & \int_{\Omega}\delta\left(x-y\right)f\left(y\right)dy=f\left(x\right);
\end{eqnarray*}
where $dy=dy_{1}dy_{2}\cdots dy_{\nu}$ denotes the Lebesgue measure
in $\Omega\subset\mathbb{R}^{\nu}$.
\end{proof}
We shall now consider discrete subsets:
\begin{thm}
\label{thm:main}Let $\Omega$, and $V\subset\Omega$, be given. Then
\begin{enumerate}
\item Discrete case: Fix $V\subset\Omega$, $\#V=\aleph_{0}$, where $V=\left\{ x_{j}\right\} _{j=1}^{\infty}$,
$x_{j}\in\Omega$. Let 
\[
\mathscr{H}\left(V\right)=\text{RKHS of \ensuremath{K^{\left(V\right)}:=K\big|_{V\times V}}};
\]
then $\delta_{x_{j}}\in\mathscr{H}\left(V\right)$. 
\item Continuous case; by contrast: $K_{x}^{\left(V\right)}\in\mathscr{H}_{CM}\left(V\right)$,
but $\delta_{x}\notin\mathscr{H}_{CM}\left(\Omega\right)$, $x\in\Omega$.
\end{enumerate}
\end{thm}
The proof will be given in the next section.

To see that $\delta_{x}\notin\mathscr{H}_{CM}\left(\Omega\right)$,
we use (\ref{eq:m3}) when $\nu=1$, and (\ref{eq:m4}) when $\nu>1$. 

In general, by elliptic regularity, $\mathscr{H}_{CM}\left(\Omega\right)$
is a RKHS of continuous functions; and $\delta_{x}$ is \emph{not
}a function, so not in $\mathscr{H}_{CM}\left(\Omega\right)$. 

But the RKHS of $K^{\left(V\right)}:=K\big|_{V\times V}$ is a \emph{discrete}
RKHS, and $\delta_{x}\in\mathscr{H}\left(V\right)$; proof below;
\thmref{main2}.

\section{\label{sec:net}Infinite network of resistors}

Here we introduce a family of positive definite kernels $k:V\times V\rightarrow\mathbb{R}$,
defined on infinite sets $V$ of vertices for a given graph $G=\left(V,E\right)$
with edges $E\subset V\times V\backslash(\text{diagonal})$.

There is a large literature dealing with analysis on infinite graphs;
see e.g., \cite{JP10,JP11,JP13}; see also \cite{OS05,BCF07,CJ11}.

Our main purpose here is to point out that every assignment of resistors
on the edges $E$ in $G$ yields a p.d. kernel $k$, and an associated
RKHS $\mathscr{H}=\mathscr{H}\left(k\right)$ such that
\begin{equation}
\delta_{x}\in\mathscr{H},\quad\text{for all \ensuremath{x\in V}.}\label{eq:g1}
\end{equation}

\begin{defn}
\label{def:g}Let $G=\left(V,E\right)$ be as above. Assume 
\begin{enumerate}[label=\arabic{enumi}.]
\item $\left(x,y\right)\in E\Longleftrightarrow\left(y,x\right)\in E$;
\item $\exists c:E\rightarrow\mathbb{R}_{+}$ (a conductance function =
1 / resistance) such that

\begin{enumerate}[label=(\roman{enumii})]
\item \label{enu:g1} $c_{\left(xy\right)}=c_{\left(yx\right)}$, $\forall\left(xy\right)\in E$; 
\item \label{enu:g2}for all $x\in V$, $\#\left\{ y\in V\:|\:c_{\left(xy\right)}>0\right\} <\infty$;
and
\item $\exists o\in V$ s.t. for $\forall x\in V\backslash\left\{ o\right\} $,
$\exists$ edges $\left(x_{i},x_{i+1}\right)_{0}^{n-1}\in E$ s.t.
$x_{o}=0$, and $x_{n}=x$; called connectedness. 
\end{enumerate}
\end{enumerate}
\end{defn}
Given $G=\left(V,E\right)$, and a fixed conductance function $c:E\rightarrow\mathbb{R}_{+}$
as specified above, we now define a corresponding Laplace operator
$\Delta=\Delta^{\left(c\right)}$ acting on functions on $V$, i.e.,
on $\mathscr{F}unc\left(V\right)$ by 
\begin{equation}
\left(\Delta f\right)\left(x\right)=\sum_{y\sim x}c_{xy}\left(f\left(x\right)-f\left(y\right)\right).\label{eq:g2}
\end{equation}
See Fig \ref{fig:res1}-\ref{fig:res2} for examples of networks of
resistors: $c_{xy}=\frac{1}{res\left(x,y\right)}$, if $\left(x,y\right)\in E$.

\begin{figure}
\begin{tabular}{cc}
\includegraphics[width=0.3\columnwidth]{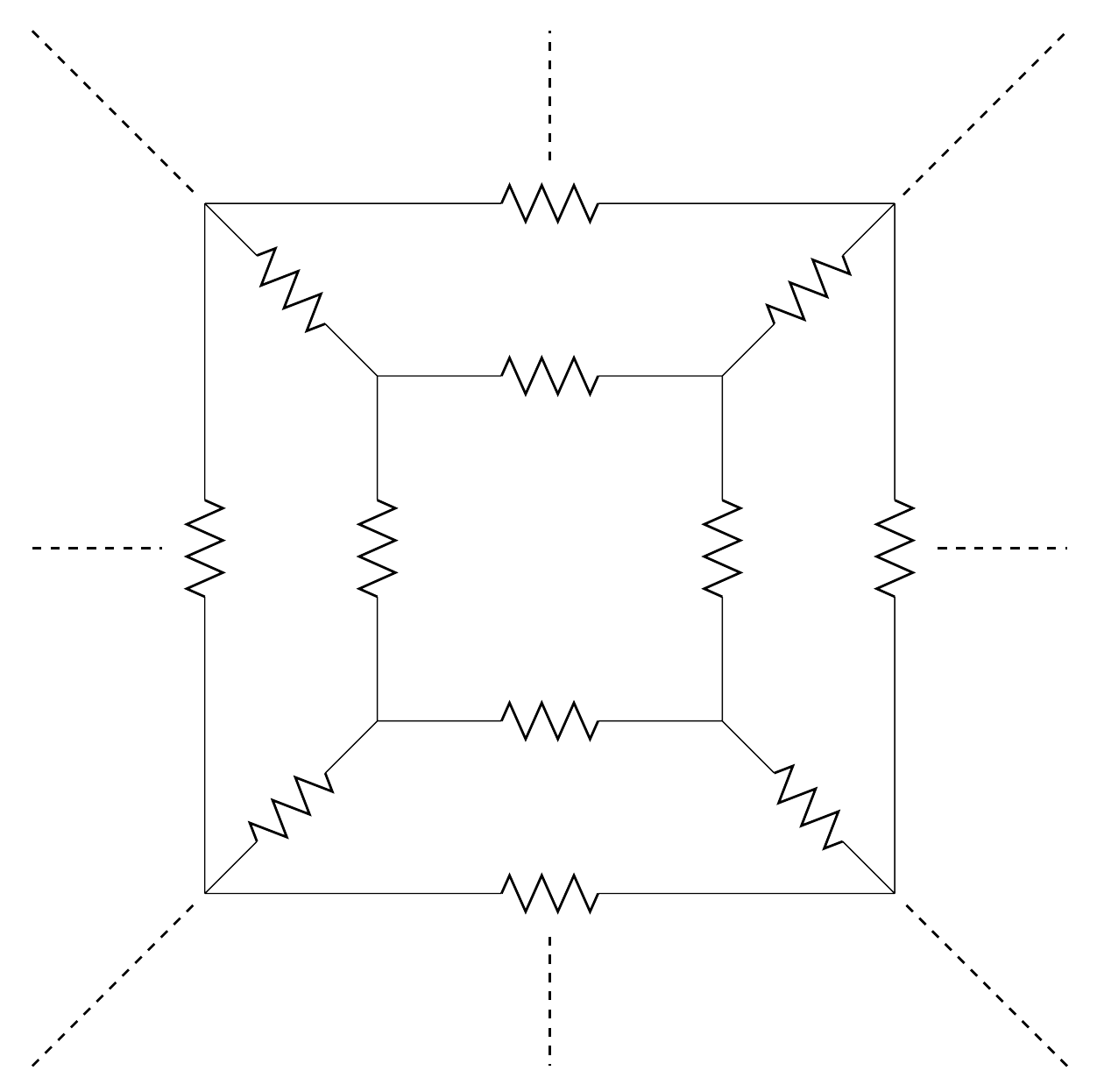} & \includegraphics[width=0.35\columnwidth]{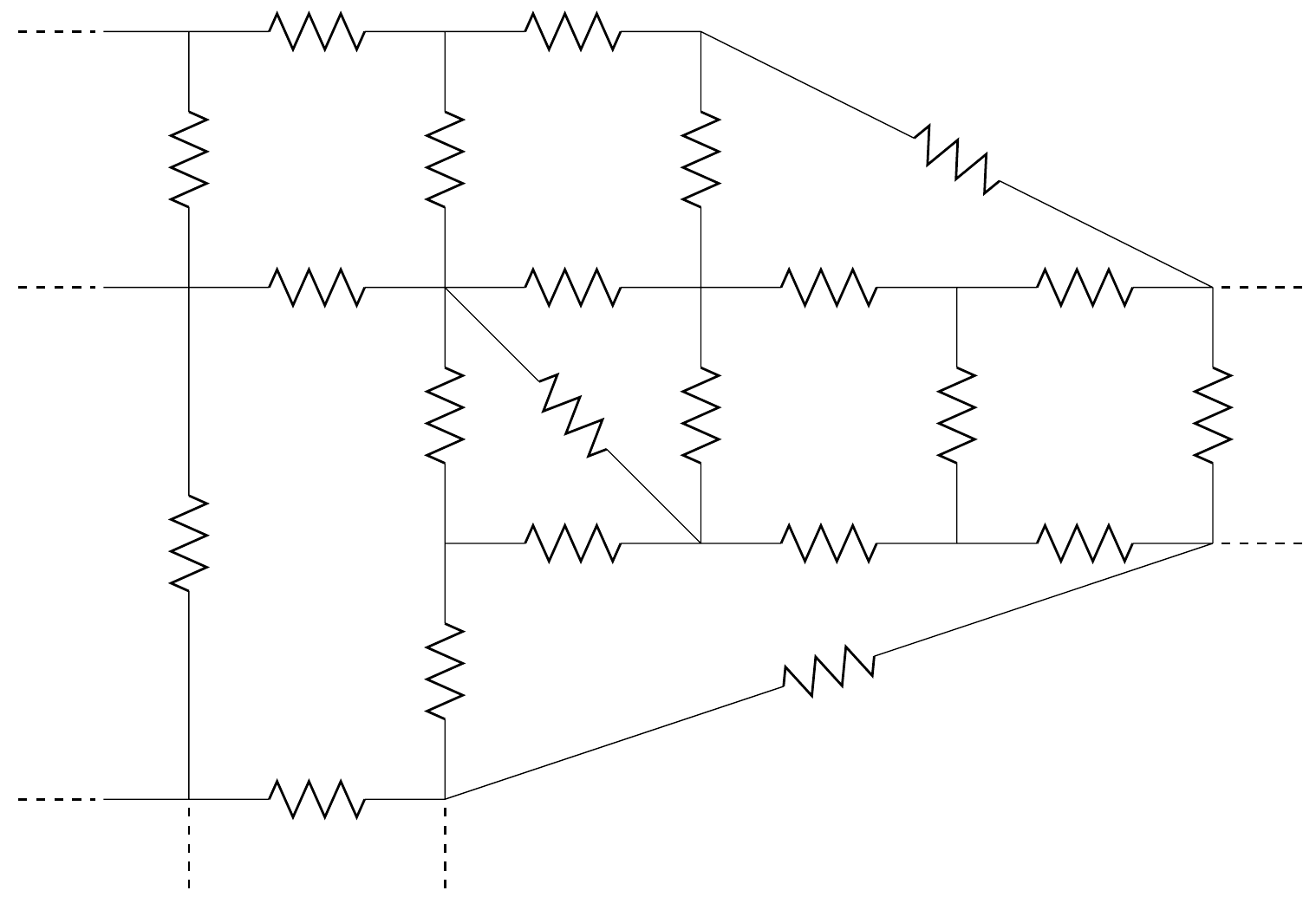}\tabularnewline
\end{tabular}

\protect\caption{\label{fig:res1}Examples of configuration of resistors in a network.}

\end{figure}

\begin{figure}
\includegraphics[width=0.6\columnwidth]{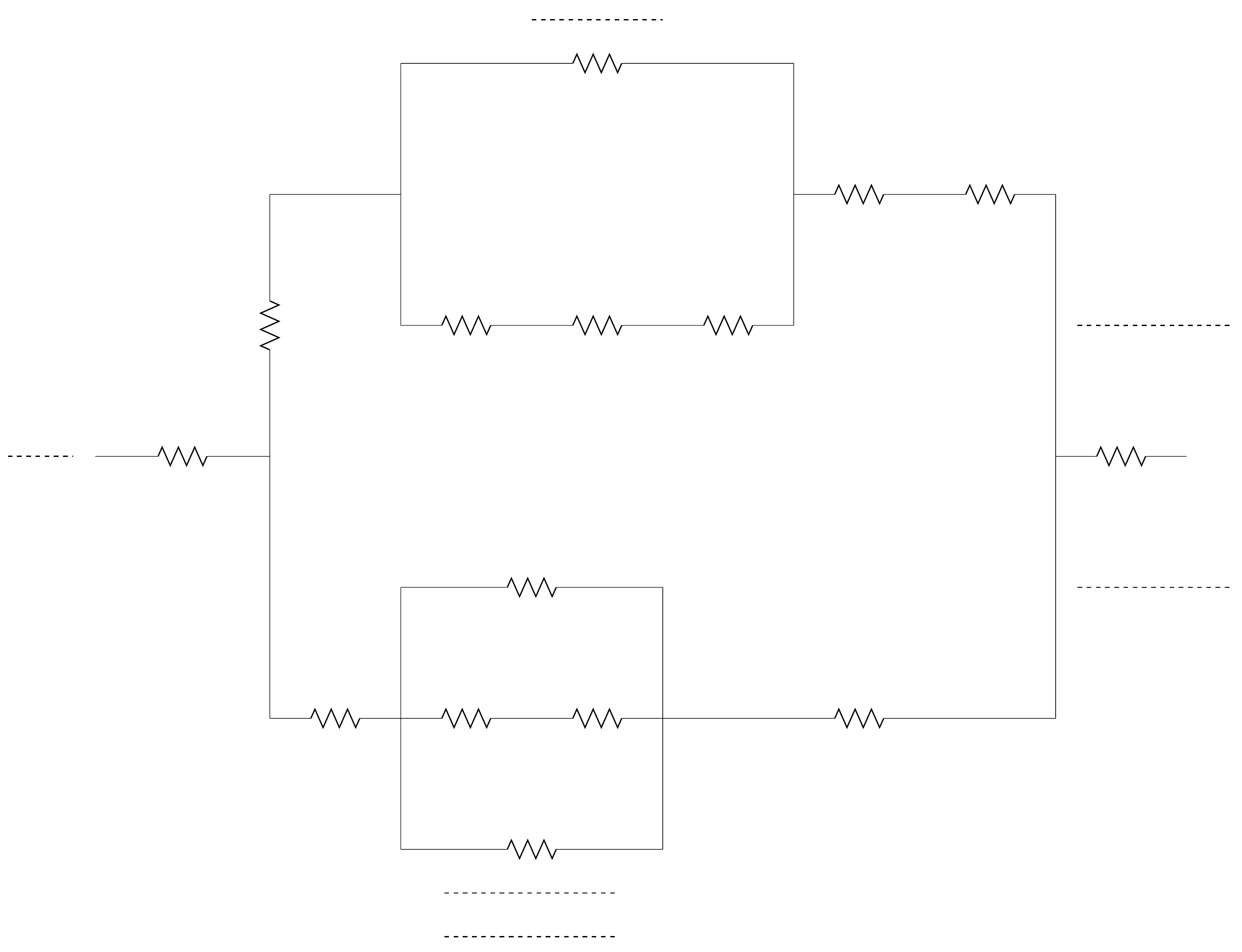}

\protect\caption{\label{fig:res2}Example of configuration of resistors in a \textquotedblleft large\textquotedblright{}
network.}
\end{figure}

Let $\mathscr{H}$ be the Hilbert space defined as follows: A function
$f$ on $V$ is in $\mathscr{H}$ iff $f\left(o\right)=0$, and 
\begin{equation}
\left\Vert f\right\Vert _{\mathscr{H}}^{2}:=\frac{1}{2}\underset{\underset{\subset V\times V}{\left(x,y\right)\in E}}{\sum\sum}c_{xy}\left|f\left(x\right)-f\left(y\right)\right|^{2}<\infty.\label{eq:g3}
\end{equation}

\begin{lem}[\cite{JP10}]
\label{lem:lap2}For all $x\in V\backslash\left\{ o\right\} $, $\exists v_{x}\in\mathscr{H}$
s.t. 
\begin{equation}
f\left(x\right)-f\left(o\right)=\left\langle v_{x},f\right\rangle _{\mathscr{H}},\quad\forall f\in\mathscr{H}\label{eq:g4}
\end{equation}
where 
\begin{equation}
\left\langle h,f\right\rangle _{\mathscr{H}}=\frac{1}{2}\underset{\left(x,y\right)\in E}{\sum\sum}c_{xy}\left(\overline{h\left(x\right)}-\overline{h\left(y\right)}\right)\left(f\left(x\right)-f\left(y\right)\right),\quad\forall h,f\in\mathscr{H}.\label{eq:g41}
\end{equation}
(The system $\left\{ v_{x}\right\} $ is called a system of \uline{dipoles}.)\end{lem}
\begin{proof}
Let $x\in V\backslash\left\{ o\right\} $, and use (\ref{eq:g2})
together with the Schwarz-inequality to show that 
\[
\left|f\left(x\right)-f\left(o\right)\right|^{2}\leq\sum_{i}\frac{1}{c_{x_{i}x_{i+1}}}\sum_{i}c_{x_{i}x_{i+1}}\left|f\left(x_{i}\right)-f\left(x_{i+1}\right)\right|^{2}.
\]
An application of Riesz' lemma then yields the desired conclusion. 

Note that $v_{x}=v_{x}^{\left(c\right)}$ in (\ref{eq:g4}) depends
on the choice of base point $o\in V$, and on conductance function
$c$; see \ref{enu:g1}-\ref{enu:g2} and (\ref{eq:g3}).
\end{proof}
The \emph{resistance metric} $R^{\left(c\right)}\left(x,y\right)=res\left(x,y\right)$
is as follows:
\begin{align*}
R^{\left(c\right)}\left(x,y\right) & =\inf\left\{ K\::\:\left|f\left(x\right)-f\left(y\right)\right|^{2}\leq K\left\Vert f\right\Vert _{\mathscr{H}}^{2},\:\forall f\in\mathscr{H}\right\} \\
 & =\sup\left\{ \left|f\left(x\right)-f\left(y\right)\right|\::\:f\in\mathscr{H},\:\left\Vert f\right\Vert _{\mathscr{H}}\leq1\right\} .
\end{align*}

Now set 
\begin{equation}
k^{\left(c\right)}\left(x,y\right)=\left\langle v_{x},v_{y}\right\rangle _{\mathscr{H}},\quad\forall\left(xy\right)\in\left(V\backslash\left\{ o\right\} \right)\times\left(V\backslash\left\{ o\right\} \right).\label{eq:g5}
\end{equation}
It follows from a theorem that $k^{\left(c\right)}$ is a Green's
function for the Laplacian $\Delta^{\left(c\right)}$ in the sense
that 
\begin{equation}
\Delta^{\left(c\right)}k^{\left(c\right)}\left(x,\cdot\right)=\delta_{x}\label{eq:g6}
\end{equation}
where the dot in (\ref{eq:g6}) is the dummy-variable in the action.
(Note that the solution to (\ref{eq:g6}) is not unique.)

Finally, we note that 
\begin{equation}
\Delta v_{x}=\delta_{x}-\delta_{o},\quad\forall x\in V\backslash\left\{ o\right\} .\label{eq:g11}
\end{equation}
And (\ref{eq:g11}) in turn follows from (\ref{eq:g4}), (\ref{eq:g2})
and a straightforward computation.
\begin{cor}
\label{cor:lap1}Let $G=\left(V,E\right)$ and conductance $c:E\rightarrow\mathbb{R}_{+}$
be as specified above. Let $\Delta=\Delta^{\left(c\right)}$ be the
corresponding Laplace operator. Let $\mathscr{H}=\mathscr{H}\left(k^{c}\right)$
be the RKHS. Then 
\begin{equation}
\left\langle \delta_{x},f\right\rangle _{\mathscr{H}}=\left(\Delta f\right)\left(x\right)\label{eq:g12}
\end{equation}
and 
\begin{equation}
\delta_{x}=c\left(x\right)v_{x}-\sum_{y\sim x}c_{xy}v_{y}\in\mathscr{H}\label{eq:g121}
\end{equation}
holds for all $x\in V$. \end{cor}
\begin{proof}
Since the system $\left\{ v_{x}\right\} $ of dipoles (see (\ref{eq:g4}))
span a dense subspace in $\mathscr{H}$, it is enough to verify (\ref{eq:g12})
when $f=v_{y}$ for $y\in V\backslash\left\{ o\right\} $. But in
this case, (\ref{eq:g12}) follows from (\ref{eq:g6}) and a direct
calculation. (For details, see \cite{JT15}.)\end{proof}
\begin{cor}
\label{cor:lap4}Let $G=\left(V,E\right)$, and conductance $c:E\rightarrow\mathbb{R}_{+}$
be as before; let $\Delta^{\left(c\right)}$ be the Laplace operator,
and $\mathscr{H}_{E}^{\left(c\right)}$ the energy-Hilbert space in
Definition \ref{def:g} (see (\ref{eq:g3})). Let $k^{\left(c\right)}\left(x,y\right)=\left\langle v_{x},v_{y}\right\rangle _{\mathscr{H}_{E}}$
be the kernel from (\ref{eq:g5}), i.e., the Green's function of $\Delta^{\left(c\right)}$.
Then the two Hilbert spaces $\mathscr{H}_{E}$, and $\mathscr{H}\left(k^{\left(c\right)}\right)=RKHS\left(k^{\left(c\right)}\right)$,
are naturally isometrically isomorphic via $v_{x}\longmapsto k_{x}^{\left(c\right)}$
where $k_{x}^{\left(c\right)}=k^{\left(c\right)}\left(x,\cdot\right)$
for all $x\in V$. \end{cor}
\begin{proof}
Let $F\in\mathscr{F}\left(V\right)$, and let $\xi$ be a function
on $F$; then 
\begin{eqnarray*}
\left\Vert \sum\nolimits _{x\in F}\xi\left(x\right)k_{x}^{\left(c\right)}\right\Vert _{\mathscr{H}\left(k^{\left(c\right)}\right)}^{2} & = & \underset{F\times F}{\sum\sum}\overline{\xi\left(x\right)}\xi\left(y\right)k^{\left(c\right)}\left(x,y\right)\\
 & \underset{\left(\ref{eq:g5}\right)}{=} & \underset{F\times F}{\sum\sum}\overline{\xi\left(x\right)}\xi\left(y\right)\left\langle v_{x},v_{y}\right\rangle _{\mathscr{H}_{E}}\\
 & = & \left\Vert \sum\nolimits _{x\in F}\xi\left(x\right)v_{x}\right\Vert _{\mathscr{H}_{E}}^{2}.
\end{eqnarray*}

The remaining steps in the proof of the Corollary now follow from
the standard completion from dense subspaces in the respective two
Hilbert spaces $\mathscr{H}_{E}$ and $\mathscr{H}\left(k^{\left(c\right)}\right)$. 
\end{proof}
In the following we show how the kernels $k^{\left(c\right)}:V\times V\rightarrow\mathbb{R}$
from (\ref{eq:g5}) in \lemref{lap2} are related to metrics on $V$;
so called \emph{resistance metrics} (see, e.g., \cite{JP10,AJSV13}.)
\begin{cor}
\label{cor:lap2}Let $G=\left(V,E\right)$, and conductance $c:E\rightarrow\mathbb{R}_{+}$
be as above; and let $k^{\left(c\right)}\left(x,y\right):=\left\langle v_{x},v_{y}\right\rangle _{\mathscr{H}_{E}}$
be the corresponding Green's function for the graph Laplacian $\Delta^{\left(c\right)}$. 

Then there is a \uline{metric} $R\left(=R^{\left(c\right)}=\mbox{the resistance metric}\right)$,
such that 
\begin{equation}
k^{\left(c\right)}\left(x,y\right)=\frac{R^{\left(c\right)}\left(o,x\right)+R^{\left(c\right)}\left(o,y\right)-R^{\left(c\right)}\left(x,y\right)}{2}\label{eq:gm1}
\end{equation}
holds on $V\times V$. Here the base-point $o\in V$ is chosen and
fixed s.t. 
\begin{equation}
\left\langle v_{x},f\right\rangle _{\mathscr{H}_{E}}=f\left(x\right)-f\left(o\right),\quad\forall f\in\mathscr{H}_{E},\;\forall x\in V.\label{eq:gm2}
\end{equation}
\end{cor}
\begin{proof}
See \cite{JP10}. Set 
\begin{equation}
R^{\left(c\right)}\left(x,y\right)=\left\Vert v_{x}-v_{y}\right\Vert _{\mathscr{H}_{E}}^{2}.\label{eq:gm3}
\end{equation}
We proved in \cite{JP10} that $R^{\left(c\right)}\left(x,y\right)$
in (\ref{eq:gm3}) indeed defines a \emph{metric} on $V$; the so
called \emph{resistance metric}. It represents the voltage-drop from
$x$ to $y$ when 1 Amp is fed into $\left(G,c\right)$ at the point
$x$, and then extracted at $y$. 

The verification of (\ref{eq:gm1}) is now an easy computation, as
follows:
\begin{eqnarray*}
 &  & \frac{R^{\left(c\right)}\left(o,x\right)+R^{\left(c\right)}\left(o,y\right)-R^{\left(c\right)}\left(x,y\right)}{2}\\
 & = & \frac{\left\Vert v_{x}\right\Vert _{\mathscr{H}_{E}}^{2}+\left\Vert v_{y}\right\Vert _{\mathscr{H}_{E}}^{2}-\left\Vert v_{x}-v_{y}\right\Vert _{\mathscr{H}_{E}}^{2}}{2}\\
 & = & \left\langle v_{x},v_{y}\right\rangle _{\mathscr{H}_{E}}\\
 & = & k^{\left(c\right)}\left(x,y\right)\quad\text{(by \ensuremath{\left(\ref{eq:g5}\right)})}.
\end{eqnarray*}
\end{proof}
\begin{cor}
The functions $R^{\left(c\right)}\left(\cdot,\cdot\right)$ which
arise as in (\ref{eq:gm1}) and (\ref{eq:gm3}) are \emph{conditionally
negative definite}, i.e., for all finite subsets $F\subset V$ and
functions $\xi$ on $F$, such that $\sum_{x\in F}\xi_{x}=0$, we
have:
\begin{equation}
\sum_{F}\sum_{F}\overline{\xi}_{x}\xi_{y}R^{\left(c\right)}\left(x,y\right)\leq0.\label{eq:R1}
\end{equation}
(Note that (\ref{eq:R1}) follows from (\ref{eq:gm3}).)

Moreover, if $R^{\left(c\right)}$ arises as a restriction of a metric
on $\mathbb{R}^{\nu}$, then there is a quadratic form $Q$ on $\mathbb{R}^{\nu}$
(possibly $Q=0$), and a positive measure $\mu$ on $\mathbb{R}^{\nu}$
such that
\begin{equation}
R^{\left(c\right)}\left(o,x\right)=Q\left(x\right)+\int_{\mathbb{R}^{\nu}}\frac{1-\cos\left(x\cdot\xi\right)}{\left|\xi\right|^{2}}d\mu\left(\xi\right),\label{eq:R2}
\end{equation}
where
\begin{equation}
\int_{\mathbb{R}^{\nu}}\frac{d\mu\left(\xi\right)}{1+\left|\xi\right|^{2}}<\infty.\label{eq:R3}
\end{equation}
\end{cor}
\begin{proof}
For details on this last point, see for example \cite{AJV14} and
\cite{BTA04}. \end{proof}
\begin{prop}
In the two cases: (i) $B\left(t\right)$, Brownian motion on $0<t<\infty$;
and (ii) the Brownian bridge $B_{bri}\left(t\right)$, $0<t<1$, (see
Fig \ref{fig:bb}) the corresponding resistance metric $R$ is as
follows:
\begin{enumerate}
\item[(i)]  If $V=\left\{ x_{i}\right\} _{i=1}^{\infty}\subset\left(0,\infty\right)$,
$x_{1}<x_{2}<\cdots$, then 
\begin{equation}
R_{B}^{\left(V\right)}\left(x_{i},x_{j}\right)=\left|x_{i}-x_{j}\right|.
\end{equation}

\item[(ii)]  If $W=\left\{ x_{i}\right\} _{i=1}^{\infty}\subset\left(0,1\right)$,
$0<x_{1}<x_{2}<\cdots<1$, then 
\begin{equation}
R_{bridge}^{\left(W\right)}\left(x_{i},x_{j}\right)=\left|x_{i}-x_{j}\right|\cdot\left(1-\left|x_{i}-x_{j}\right|\right).
\end{equation}
In the completion w.r.t. the resistance metric $R_{bridge}^{\left(W\right)}$,
the two endpoints $x=0$ and $x=1$ are identified; see also Fig \ref{fig:bb}.
\end{enumerate}
\end{prop}
\begin{figure}
\includegraphics[width=0.5\columnwidth]{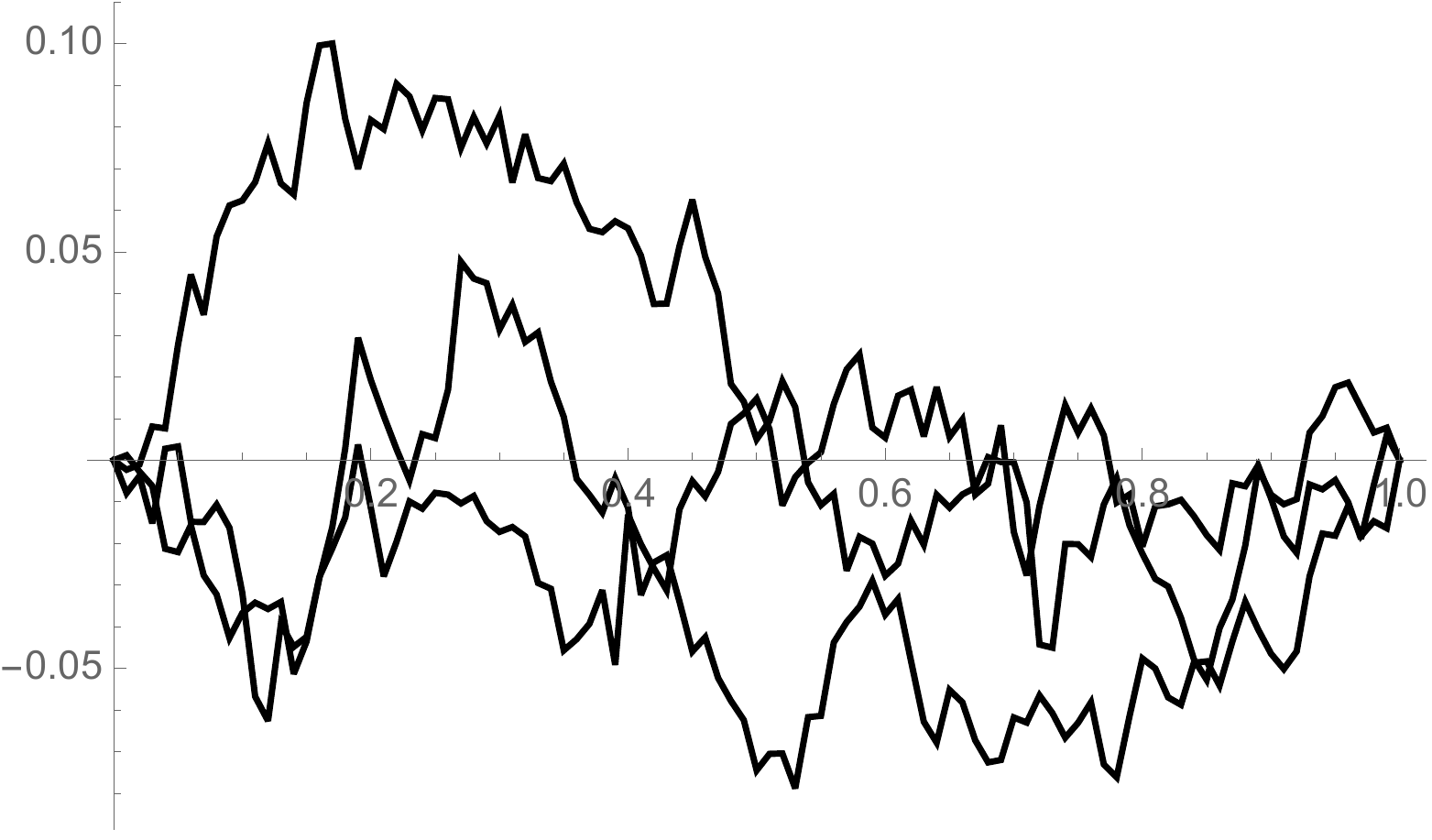}

\protect\caption{\label{fig:bb}Brownian bridge $B_{bri}\left(t\right)$, a simulation
of three sample paths of the Brownian bridge.}
\end{figure}

The Brownian bridge $B_{bri}\left(t\right)$ is realized on a probability
space $\Omega\left(\simeq C\left(\left[0,1\right]\right)\right)$
such that $B_{bri}\left(0\right)=B_{bri}\left(1\right)=0$, and 
\begin{equation}
\mathbb{E}\left(B_{bri}\left(s\right)B_{bri}\left(t\right)\right)=s\wedge t-st=k_{B}\left(s,t\right),\label{eq:bb1}
\end{equation}
where $\mathbb{E}\left(\cdots\right)=\int_{\Omega}\cdots dP$, and
$P$ denotes Wiener measure on $\Omega$.

If $\left\{ B\left(t\right)\right\} _{t\in\in\left[0,1\right]}$ denotes
the usual Brownian motion, $B\left(0\right)=0$, with covariance 
\begin{equation}
\mathbb{E}\left(B\left(s\right)B\left(t\right)\right)=s\wedge t;\label{eq:bb2}
\end{equation}
then we may take for $B_{bri}\left(t\right)$ as follows:
\begin{equation}
B_{bri}\left(t\right)=\left(1-t\right)B\left(\frac{t}{1-t}\right),\quad\forall t\in\left(0,1\right).\label{eq:bb3}
\end{equation}

\section{\label{sec:drkhsv}The Discrete RKHSs $\mathscr{H}\left(V\right)$
from Brownian motion }

Let $V$, $K$ be as above. To get that $\delta_{x_{i}}\in\mathscr{H}\left(V\right)$,
we may specify a graph $G=\left(V,E\right)$ with vertices $V$ and
edges $E\subset V\times V\backslash\left\{ \text{diagonal}\right\} $,
and assume that, for all $x_{i}\in V$, 
\begin{equation}
\#\left\{ x_{j}\:\big|\:x_{i}\sim x_{j}\right\} <\infty,\quad\forall x_{i}\in V\label{eq:L0}
\end{equation}
(finite neighborhood set); see Fig \ref{fig:fnbh}. Fix a base-point
$o\in V$.

\begin{figure}
\includegraphics[width=0.5\columnwidth]{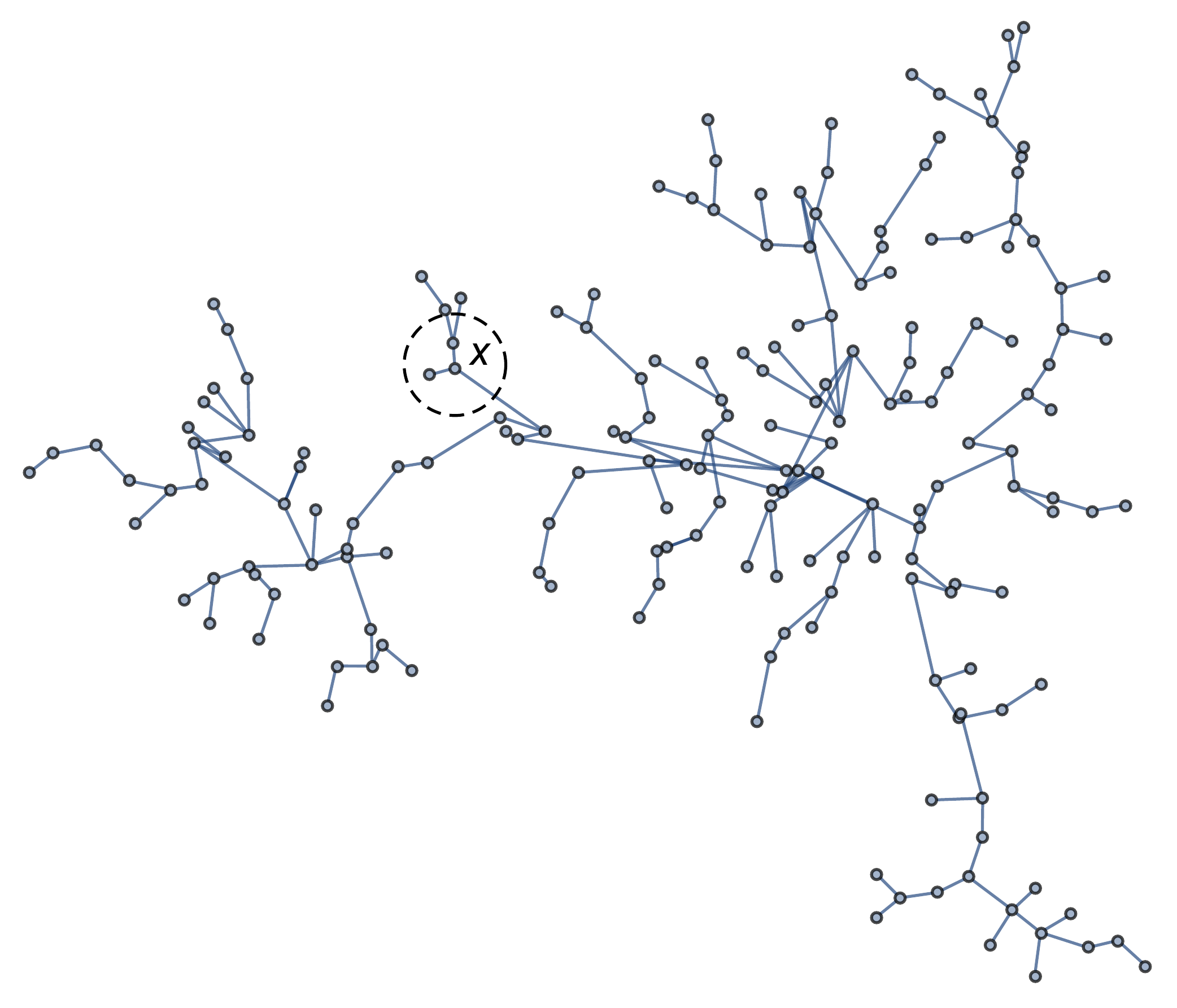}

\protect\caption{\label{fig:fnbh}$\forall i$, $\#\left\{ j\:s.t.\:\left(x_{i}x_{j}\right)\in E\right\} <\infty$,
every vertex $x_{i}$ has a finite set of neighbors.}
\end{figure}

Set $\mathscr{H}_{E}:=$ all functions $f$ on $V$ s.t. $f\left(o\right)=0$,
with
\[
\left\Vert f\right\Vert _{E}^{2}=\frac{1}{2}\underset{i\sim j}{\sum\sum}\frac{1}{res\left(i,j\right)}\left|f\left(x_{i}\right)-f\left(x_{j}\right)\right|^{2};
\]
where $res\left(i,j\right)$ denotes resistance between $x_{i}$ and
$x_{j}$, and 
\[
c_{ij}:=\frac{1}{res\left(i,j\right)}
\]
is the conductance. Then $\mathscr{H}_{E}\left(G\right)$ is a RKHS
for the graph $G=\left(V,E\right)$, i.e., the energy Hilbert space.

Note that 
\[
\widetilde{K}_{x_{i}}:=K\left(x_{i},\cdot\right)\Big|_{V}\in\mathscr{H}_{E}\left(G\right)
\]
and 
\[
\left\langle \widetilde{K}_{x_{i}},f\right\rangle _{\mathscr{H}_{E}}=f\left(x_{i}\right),\quad\forall x_{i}\in V,\;\forall f\in\mathscr{H}_{E}\left(G\right).
\]

\begin{lem}
The mapping $P_{G}:CM\left(\Omega\right)\rightarrow\mathscr{H}_{E}\left(G\right)$,
$K_{x_{i}}\longmapsto\widetilde{K}_{x_{i}}$, defines a projection
of the Cameron-Martin space (see (\ref{eq:m4})) onto $\mathscr{H}_{E}\left(G\right)$.\end{lem}
\begin{proof}
Note that $J_{V}:\mathscr{H}_{E}\left(G\right)\rightarrow CM\left(\Omega\right)$,
$J_{V}\widetilde{K}_{x_{j}}=K_{x_{j}}$, is an isometry. In fact,
we have that 
\[
\left\Vert \sum\nolimits _{j}\xi_{j}\widetilde{K}_{x_{j}}\right\Vert _{\mathscr{H}_{E}\left(G\right)}^{2}=\left\Vert \sum\nolimits _{j}\xi_{j}K_{x_{j}}\right\Vert _{CM\left(\Omega\right)}^{2};
\]
where $\left\Vert f\right\Vert _{CM}^{2}=\int_{\Omega}\left|\nabla f\right|^{2}dx$,
for all $f\in CM\left(\Omega\right)$. Now, $P_{G}=J_{V}J_{V}^{*}$. \end{proof}
\begin{cor}
$CM\left(\Omega\right)\ominus\mathscr{H}_{E}\left(G\right)=\left\{ f\in CM\left(\Omega\right)\:;\:f\left(x_{j}\right)=0,\;\forall x_{j}\in V\right\} $
(See also \lemref{sp}.)\end{cor}
\begin{proof}
The same as in the proof for the case of Brownian bridge: 
\[
\left\langle J\widetilde{K}_{x_{j}},f\right\rangle _{CM\left(\Omega\right)}=f\left(x_{j}\right),\quad\forall f\in CM\left(\Omega\right).
\]
The desired result follows from this.\end{proof}
\begin{defn}
Set
\begin{equation}
\begin{split}res\left(x_{i},x_{j}\right) & =R^{\left(c\right)}\left(x_{i},x_{j}\right)\\
 & =\left\Vert K_{x_{i}}-K_{x_{j}}\right\Vert _{\mathscr{H}_{E}}^{2}\\
 & =K\left(x_{i},x_{i}\right)+K\left(x_{j},x_{j}\right)-2K\left(x_{i},x_{j}\right)
\end{split}
\label{eq:L1}
\end{equation}
(We proved in \cite{JP10} that $res\left(x_{i},x_{j}\right)$ in
(\ref{eq:L1}) indeed defines a metric on $V$; the so called \emph{resistance
metric}.) Let 
\[
\left(\Delta_{V}f\right)\left(x_{i}\right):=\sum_{x_{j}\sim x_{i}}\frac{1}{res\left(x_{i},x_{j}\right)}\left(f\left(x_{i}\right)-f\left(x_{j}\right)\right)
\]
be the graph-Laplacian; where $x_{j}\sim x_{j}$ iff $\left(x_{i}x_{j}\right)\in E$. \end{defn}
\begin{thm}
\label{thm:main2}Let $V\subset\Omega$ and $\Delta_{0}$, $K$, $\mathscr{H}\left(V\right)$
be as above; assume (\ref{eq:L0}), i.e., finite neighbors in $G$.
Then $\delta_{x_{i}}\in\mathscr{H}\left(V\right)$, and 
\[
\left(\Delta_{V}f\right)\left(x_{i}\right)=\left\langle \delta_{x_{i}},f\right\rangle _{\mathscr{H}\left(V\right)},\quad\forall f\in\mathscr{H}\left(V\right).
\]
\end{thm}
\begin{proof}
Follows from earlier analysis. One shows that 
\[
\delta_{x_{i}}=\left(\sum_{x_{l}\sim x_{i}}\frac{1}{res\left(x_{i},x_{l}\right)}\right)K_{x_{i}}-\sum_{x_{j}\sim x_{i}}\frac{1}{res\left(x_{i},x_{j}\right)}K_{x_{j}}\in\mathscr{H}\left(V\right).
\]
(It is a finite sum based on the assumption (\ref{eq:L0}).)\end{proof}
\begin{rem}
Let $\Omega\subset\mathbb{R}^{\nu}$ as above. If $\nu>1$, then the
kernel $K\left(s,t\right)$, $\left(s,t\right)\in\Omega\times\Omega$,
has a singularity at $x=y$, by contrast to $\nu=1$ (see below.)
But we can still construct discrete graph Laplacians. 
\end{rem}
Fix $\Omega$, and let $K$ be the kernel s.t. 
\begin{align*}
 & K\left(x,\cdot\right)\big|_{\partial\Omega}\equiv0,\quad\forall x\in\Omega,\\
 & \Delta K=\delta\left(s,y\right)=\delta\left(x-y\right),\quad\left(x,y\right)\in\Omega\times\Omega
\end{align*}
where $\Delta=-\nabla^{2}$ as a s.a. operator on $L^{2}\left(\Omega\right)$,
with 
\[
dom\left(\Delta\right)=\left\{ f\in L^{2}\left(\Omega\right)\:\big|\:\Delta f\in L^{2}\left(\Omega\right),\:f\big|_{\partial\Omega}=0\right\} .
\]
So $K$ satisfies the Dirichlet boundary condition, but $K$ has a
singularity at $x=y$, i.e., at the diagonal of $\Omega\times\Omega$
if $\nu>1$. 

Fix $V=\left\{ x_{i}\right\} _{1}^{\infty}\subset\Omega$, discrete.
Fix $E\subset V\times V\backslash\left(\text{diagonal}\right)$ s.t.
$\forall x_{i}\in V$, $\#\left\{ j\:\big|\:x_{j}\sim x_{i}\right\} <\infty$,
finite neighbor sets. Set 
\[
\left(\Delta_{V}f\right)\left(x_{i}\right)=\sum_{\underset{x_{i}\sim x_{j}}{x_{j}}}\frac{1}{res\left(x_{i},x_{j}\right)}\left(f\left(x_{i}\right)-f\left(x_{j}\right)\right)
\]
and get the corresponding energy-Hilbert space
\[
\mathscr{H}_{E}\left(V\right)=\left\{ f\;\text{on \ensuremath{V}}\:\big|\:\underset{x_{i}\sim x_{j}}{\sum\sum}\frac{1}{res\left(x_{i},x_{j}\right)}\left|f\left(x_{i}\right)-f\left(x_{j}\right)\right|^{2}<\infty\right\} .
\]

\begin{example}
\label{exa:bb}For $\nu=1$, let 
\begin{equation}
\begin{split} & k=s\wedge t-st\\
 & k_{s}\left(t\right):=k\left(s,t\right).
\end{split}
\label{eq:dk1}
\end{equation}
Then, 
\begin{equation}
\begin{split}k_{s}'\left(t\right) & =\left(1-s\right)\chi_{\left[0,s\right]}\left(t\right)-s\chi_{\left[s,1\right]}\left(t\right)\\
-k_{s}''\left(t\right) & =\delta_{s}\left(t\right)=\delta_{s,t}
\end{split}
\label{eq:dk2}
\end{equation}

\begin{figure}
\begin{tabular}{c}
\includegraphics[width=0.35\columnwidth]{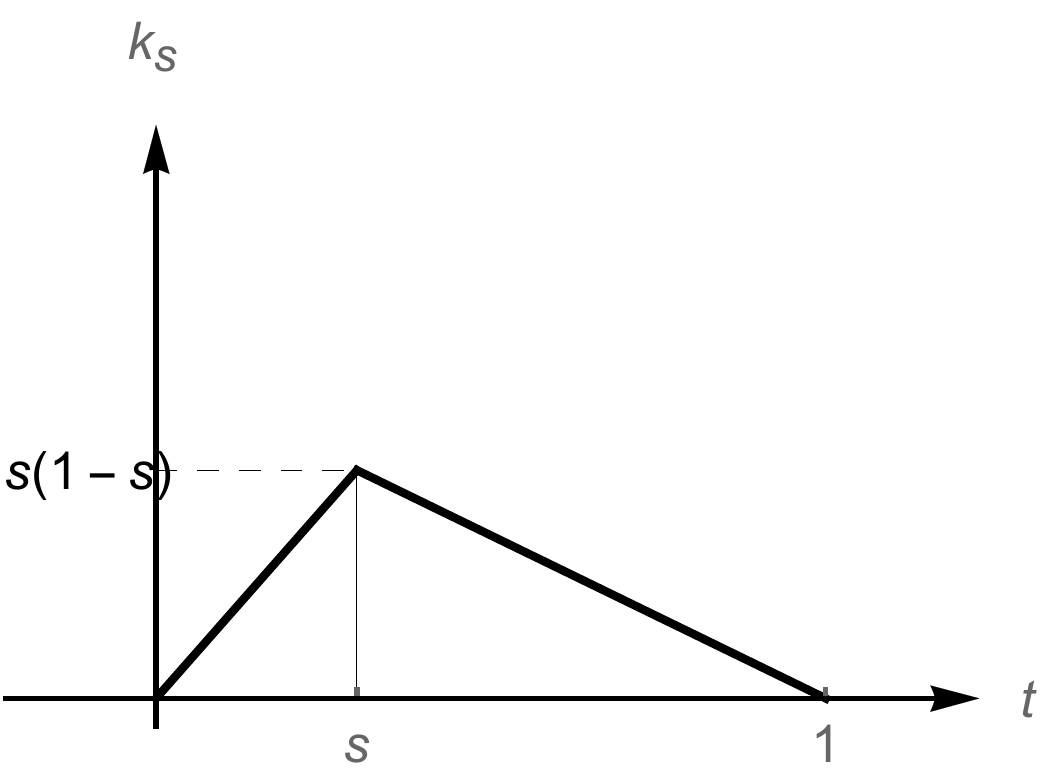}\tabularnewline
\includegraphics[width=0.35\columnwidth]{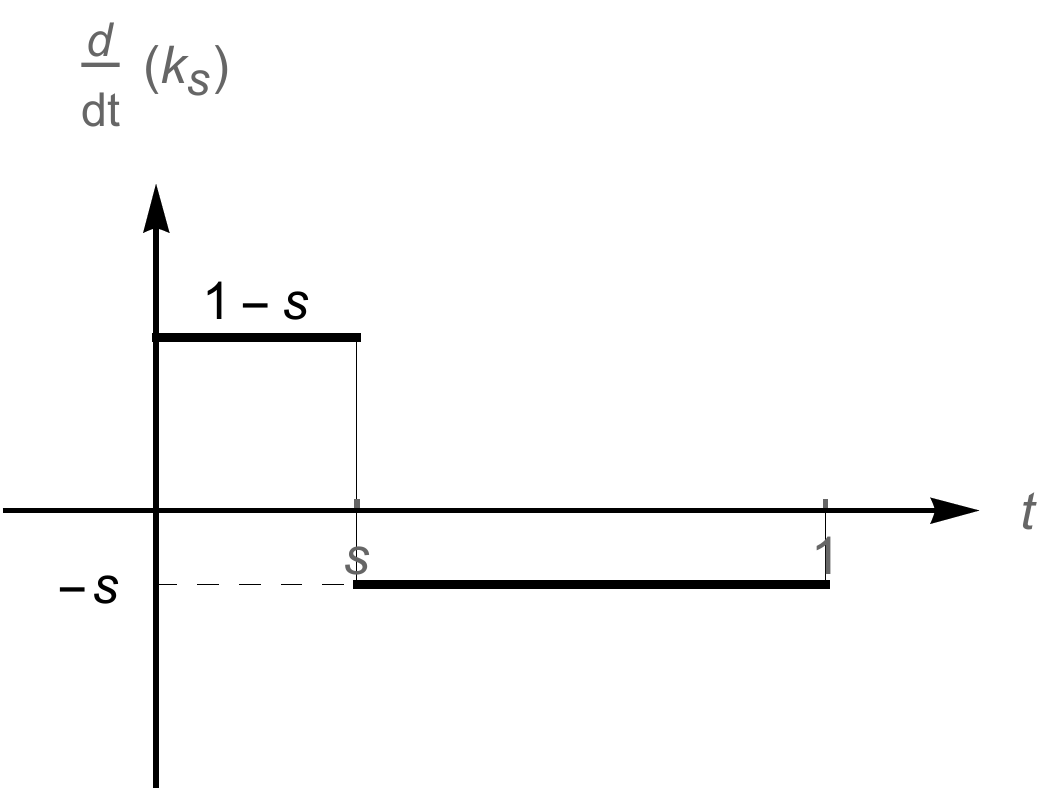}\tabularnewline
\end{tabular}

\protect\caption{\label{fig:bbf}$\nu=1$, $k_{s}\left(t\right)=k\left(s,t\right)$.
(A kernel function and its derivative.)}

\end{figure}

\end{example}

\begin{example}
For $\nu>1$, set $\Delta=-\nabla^{2}=\sum_{1}^{\nu}\left(\frac{\partial}{\partial x_{j}}\right)^{2}$,
where 
\begin{equation}
\nabla=grad=\left(\frac{\partial}{\partial x_{1}},\cdots,\frac{\partial}{\partial x_{\nu}}\right).\label{eq:dk3}
\end{equation}
Set 
\begin{equation}
\begin{split}C_{\nu} & =\begin{cases}
\frac{1}{2\pi} & \nu=2\\
\frac{\nu-2}{\left|S^{\nu-1}\right|} & \nu>2
\end{cases}\\
\Omega & =\left\{ x\in\mathbb{R}^{\nu}\:\big|\:\left|x\right|<1\right\} \\
S^{\nu-1} & =\left\{ x\in\mathbb{R}^{\nu}\:\big|\:\left|x\right|^{2}=\sum\nolimits _{j=1}^{\nu}x_{j}^{2}=1\right\} 
\end{split}
\label{eq:dk4}
\end{equation}
and, $\forall\left(x,y\right)\in\Omega\times\Omega$, 
\begin{equation}
K\left(x,y\right)=\begin{cases}
C_{\nu}\left\{ \frac{1}{\left|x-y\right|^{\nu-2}}-\frac{1}{\left(\left|x\right|\left|x^{*}-y\right|\right)^{\nu-2}}\right\}  & \nu>2\\
\\
\frac{1}{2\pi}\log\left(\frac{\left|x\right|\left|x^{*}-y\right|}{\left|x-y\right|}\right) & \nu=2
\end{cases}\label{eq:dk5}
\end{equation}
where in (\ref{eq:dk5}), we set 
\[
x^{*}:=\begin{cases}
\frac{x}{\left|x\right|^{2}} & x\in\Omega\backslash\left\{ 0\right\} \\
\infty & x=0.
\end{cases}
\]
For $\nu=2$, see Fig \ref{fig:ker}.

\begin{figure}
\includegraphics[width=0.3\columnwidth]{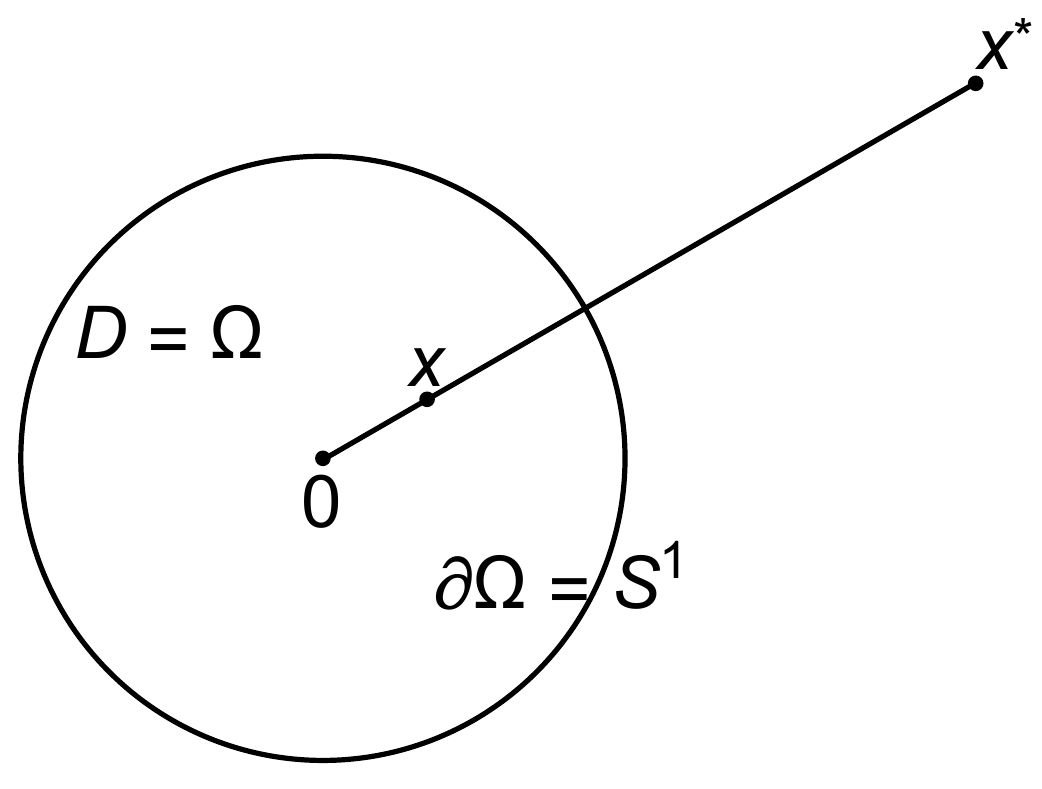}

\begin{align*}
D & =\Omega=\left\{ x\in\mathbb{R}^{\nu}\:\big|\:\left|x\right|<1\right\} \\
\partial\Omega & =S^{1}=\left\{ x\in\mathbb{R}^{\nu}\:\big|\:\left|x\right|=1\right\} 
\end{align*}

\protect\caption{\label{fig:ker}The case of $\nu=2$. $x\protect\mapsto x^{*}$, reflection
around $S^{1}$.}

\end{figure}

Special case of 
\begin{equation}
K=G-\text{Poisson}\left(G\left(x,\cdot\right)\Big|_{\partial\Omega}\right)\label{eq:dk6}
\end{equation}
where 
\begin{equation}
G\left(x,y\right)=\begin{cases}
-\frac{1}{2\pi}\log\left|x-y\right| & \nu=2\\
C_{\nu}\frac{1}{\left|x-y\right|^{\nu-2}} & \nu>2.
\end{cases}\text{ (Newton potential)}\label{eq:dk7}
\end{equation}
\end{example}
\begin{thm}
\label{thm:dk7}For $\Omega\subset\mathbb{R}^{\nu}$, bounded open,
smooth boundary $\partial\Omega$, let $m_{\partial\Omega}$ be the
surface measure (see \cite{MR2301309}), then $\Delta_{0}$ with boundary
condition $f\big|_{\partial\Omega}=0$ with 
\[
dom\Delta_{0}=\left\{ f\in L^{2}\left(\Omega\right)\:\big|\:\Delta f\in L^{2}\left(\Omega\right),\:f\big|_{\partial\Omega}=0\right\} 
\]
is s.a., and with $K$ as in (\ref{eq:dk8}), 
\[
\Delta K=\delta\left(s-y\right),\quad\left(x,y\right)\in\Omega\times\Omega
\]
with $K\left(x,\cdot\right)\big|_{\partial\Omega}=0$, $\forall x\in\Omega$. \end{thm}
\begin{proof}
For $x,y\in\Omega$, let $P_{y}$ be the Poisson kernel, and set 
\begin{equation}
K\left(x,y\right)=G\left(x,y\right)-P_{y}\left[G\left(x,\cdot\right)\Big|_{\partial\Omega}\right]\label{eq:dk8}
\end{equation}
view $G\left(x,\cdot\right)$ as a function on the boundary $\partial\Omega$.
We have 
\[
P_{y}\left[G\left(x,\cdot\right)\Big|_{\partial\Omega}\right]=\int_{\partial\Omega}P_{y}\left(b\right)G\left(x,b\right)dm_{\partial\Omega}\left(b\right),
\]
where $dm_{\partial\Omega}\left(\cdot\right)$ denotes the standard
measure on the boundary $\partial\Omega$ of $\Omega$; see \cite{MR2301309}. 

Then 
\[
\Omega\ni y\longrightarrow P_{y}\left[G\left(x,\cdot\right)\Big|_{\partial\Omega}\right]
\]
is harmonic in $\Omega$, and $\lim_{y\rightarrow b}P_{y}\left[G\left(x,\cdot\right)\Big|_{\partial\Omega}\right]=G\left(x,b\right)$.
So, from (\ref{eq:dk8}), 
\[
K\left(x,b\right)=0,\quad\forall x\in\Omega,\:b\in\partial\Omega,
\]
and so $K$ in (\ref{eq:dk8}) is the Dirichlet kernel.
\end{proof}

\begin{acknowledgement*}
The co-authors thank the following colleagues for helpful and enlightening
discussions: Professors Daniel Alpay, Sergii Bezuglyi, Ilwoo Cho,
Ka Sing Lau, Paul Muhly, Myung-Sin Song, Wayne Polyzou, Gestur Olafsson,
Keri Kornelson, and members in the Math Physics seminar at the University
of Iowa.

\bibliographystyle{amsalpha}
\bibliography{ref}
\end{acknowledgement*}

\end{document}